\documentclass[11pt,reqno]{amsart}
\usepackage{color, amsmath,amssymb, amsfonts, amstext,amsthm, latexsym}
\usepackage[active]{srcltx}
\allowdisplaybreaks
\newcommand{\RR}{{\mathbb R}}
\newcommand{\HH}{{\mathbb H}}

\newcommand{\EE}{{\mathbb E}}

\newcommand{\LL}{{\mathbb L}}
  
\newcommand{\PP}{{\mathbb P}}
\newcommand{\WW}{{\mathbb W}}
\newcommand{\VV}{{\mathbb V}}

 \newcommand{\bu}{{\bf u}}
 \newcommand{\bU}{{\bf U}}
 \newcommand{\bQ}{{\bf Q}^0_h}
 \newcommand{\bE}{{\bf E}}
\newcommand{\bPhi}{{\bf \Phi}}
\newcommand{\bG}{{\bf G}}

\numberwithin{equation}{section}
\newtheorem{theorem}{Theorem}[section]
\newtheorem{defn}[theorem]{Definition}

\newtheorem{lemma}[theorem]{Lemma}
\newtheorem{remark}[theorem]{Remark}
\newtheorem{prop}[theorem]{Proposition}
\newtheorem{coro}[theorem]{Corollary}

\begin{document}
\title[Strong Convergence for 2D NS  space-time  Euler schemes]
{Space-time  Euler discretization schemes  \\ for the stochastic 2D Navier-Stokes equations}

\author[H. Bessaih]{Hakima Bessaih}
\address{University of Wyoming, Department of Mathematics and Statistics, Dept. 3036, 1000
East University Avenue, Laramie WY 82071, United States}
\email{ bessaih@uwyo.edu}

\author[A. Millet]{ Annie Millet}
\address{SAMM, EA 4543,
Universit\'e Paris 1 Panth\'eon Sorbonne, 90 Rue de
Tolbiac, 75634 Paris Cedex France {\it and} Laboratoire de
Probabilit\'es, Statistique et Mod\'elisation, UMR 8001, 
  Universit\'es Paris~6-Paris~7} 
\email{amillet@univ-paris1.fr}

\thanks{  Hakima Bessaih was partially supported by Simons Foundation Grant 582264 }  



\begin{abstract} 
 We prove that the implicit time  Euler scheme coupled with finite elements space discretization for the 2D Navier-Stokes 
equations on the torus 
subject to a random perturbation converges in $L^2(\Omega)$, and describe the rate of convergence for an $H^1$-valued initial condition.
This refines previous results which only established the convergence in probability of these numerical approximations. Using 
exponential  moment estimates of
the solution of the stochastic Navier-Stokes equations and convergence of a   localized scheme,  we can prove strong convergence of 
this space-time approximation.
The  speed of the $L^2(\Omega)$-convergence depends on the diffusion coefficient and on the viscosity parameter. In case of Scott-Vogelius mixed
elements and for an additive noise, the convergence is polynomial.   
\end{abstract}

\maketitle

\section{Introduction}\label{s1} \smallskip

Numerical schemes and algorithms have been introduced to best approximate and construct solutions for PDEs. 
A similar approach has started to emerge for stochastic models and in particular SPDEs. Many algorithms  based on either finite difference, finite element  or spectral Galerkin methods 
(for the space discretization), and on either Euler schemes, Crank-Nicolson or Runge-Kutta schemes (for the time discretization) 
have been introduced for both the linear and nonlinear cases. Their rates of convergence have been widely investigated.
The literature on numerical analysis for SPDEs is now very extensive. 
In \cite{Ben1}  the models are either linear, have global Lipschitz properties,
 or more generally some monotonicity property. In this case the convergence is proven to be in mean square. 
When nonlinearities are involved that are not of Lipschitz or monotone type,  then a rate of  convergence 
 in mean square is more difficult to  obtain.  
  Indeed, because of the stochastic perturbation, one may not use the Gronwall lemma after taking the expectation of the
   error bound   
  since it involves a nonlinear term which is often  quadratic. 
  One way to get around this problem  is to localize the nonlinear term in order to get a linear inequality, and then use the Gronwall lemma. 
  This gives rise to  a rate of   convergence in probability, that was first introduced by J.~Printems \cite{Pri}. 
  
  In this paper, we focus on the stochastic 2D Navier-Stokes equations. Our goal is to implement a space-time discretization scheme, 
  prove its $L^2(\Omega)$ convergence, and get appropriate rates of convergence. 
  
  Our model is given by 
\begin{align} \label{2D-NS}
 \partial_t u - \nu \Delta u + (u\cdot \nabla) u + \nabla \pi & = G(u) dW\quad \mbox{\rm in } \quad (0,T)\times D,\\
 \mbox{\rm div }u&=0 \quad \mbox{\rm in } \quad (0,T)\times D,
 \end{align}
 where  $D= [0,L]^2$, $T>0$. The process 
$u: \Omega\times (0,T)\times D  \to \RR^2$  is  the velocity field 
 with initial condition $u_0\in \WW^{1,2}(D)$,  periodic boundary conditions $u(t,x+L v_i)=u(t,x)$ on  $(0,T)\times \partial D$, 
 where $v_i$, $i=1,2$
 denotes the canonical basis of $\RR^2$, and $\pi : \Omega\times (0,T)\times D  \to \RR$  is  the  pressure. 
 Here $G$ is a diffusion coefficient with global Lipschitz conditions.
Let $(\Omega, {\mathcal F}, ({\mathcal F}_t),  \PP)$ denote a filtered probability space and $W$ be a Wiener process to be precisely defined later. 
Unique global weak and strong solutions (in the PDE sense) for \eqref{2D-NS} are constructed for both additive and multiplicative noise, 
and without being exhaustive, we refer to \cite{Breckner,  ChuMil}.  

Various space-time numerical schemes have been studied for the stochastic Navier-Stokes equations with a multiplicative noise \eqref{2D-NS}.
 We refer to \cite{BrCaPr,Dor, Breckner, CarPro, BreDog}, where convergence in probability is stated with various rates of convergence. 
 As stated previously, the main tool to get the convergence in probability  is the localization of the nonlinear term over a space of large probability.
  Our previous paper \cite{Be-Mi_time} describes most of the results contained in these papers. 
  The first result on an strong, that is $L^2(\Omega)$-convergence, rate is proved in  \cite{Be-Mi_time} for an $H^1$-valued initial condition. 
  The method is based on the fact that the solution (and the scheme) have finite  moments (bounded uniformly on the mesh). 
   When the noise is additive, the solution has exponential  moments;  we used this property in \cite{Be-Mi_time} to  get an explicit polynomial strong rate
   of convergence. 
   Let us mention the result by Duan-Yang \cite{DuaYan}, where  strong convergence results are stated, and proved using semi-group 
   techniques; in that paper  the initial condition $u_0\in H^2(D)$  is more regular than in our setting. 
   \smallskip
    
In our previous paper \cite{Be-Mi_time}, only a time discretization scheme for  the model \eqref{2D-NS} was studied. 
 In the current paper, we implement a space-time discretization. More specifically, we have  an implicit time discretization
coupled with a finite elements spatial discretization. A similar algorithm has been studied in  \cite{CarPro},  where  convergence in probability was proven.
  Using the tools introduced in our previous paper \cite{Be-Mi_time},
  which are formalized in some general framework, we improve these results by proving rates of  $L^2(\Omega)$-convergence for a space-time scheme. 
 
  This paper  deals with  the  fully implicit 
   time Euler scheme with constant time mesh $\frac{T}{N}$  and the finite elements discretization with
  space parameter $h$ used in \cite{CarPro}. 
  For general finite elements, in order to ensure stability, we have to deal with the velocity and the pressure which satisfy the discrete LBB condition.
 As in the results on convergence in probability proved in \cite{CarPro},  we compare the space-time discretization $\bU^l$
 and the fully implicit time discretization $\bu^l$ in $\LL^2:=\LL^2(D)$  uniformly on the time grid $\{\l \frac{T}{N}\}_{l=0, \cdots,N}$. 
 The strong  speed of convergence we obtain 
  for the velocity  includes a term containing the pressure. 
  
  The proof relies first on a  $L^2(\Omega)$ convergence result of $\max_{1\leq l\leq N} 1_{\Omega_{l-1}(M)} |\bu^l - \bU^l|_{\LL^2}^2$ 
  localized by a set $\tilde{\Omega}_l(M)$. On that set,  some power  of the $\WW^{1,2}$-norm of the time discretization  $\bu^l$ is controlled. 
  However, due to the fact that the classical non-linear term $[u.\nabla]u$ has to be coupled with $\frac{1}{2} [{\rm div}\, u] u$, for general
  finite elements the localization set is $\tilde{\Omega}_l(M):= \big\{ \max_{j=0, \cdots, l} \|\bu^j\|_{\WW^{1,2}}^4 \leq M\big\}$.  Since we can control
  moments of the $\LL^2$-norm of both discretizations $\bu^l$ and $\bU^l$, choosing the bound $M(N,h)$ in an appropriate way, we can prove
  a strong speed of convergence. In general, the error  contains a term including the pressure and its discretization. To get more precise
  results, we have to make some assumptions either on the pressure  or on the relation between $h$ and $T/N$, or to strengthen the assumptions
  on the diffusion coefficient.
  In case of a multiplicative stochastic perturbation, 
   for a constant divergence-free $\WW^{1,2}$-valued initial condition (or more generally for a Gaussian $\WW^{1,2}$-valued initial random variable), 
    the strong speed of 
  convergence of $\max_l |\bu^l - \bU^l|_{\LL^2}^2 $ is any negative power of $\ln( h^2 + T/N)$.  If the noise is additive the speed of
  convergence is $C \exp( c \sqrt{\ln(h^2 + T/N)})$. The fact that the convergence is much faster is related to the existence of exponential
  moments of $\max_{0\leq l\leq N} \|\bu^l\|_{\WW^{1,2}}^2$. However, the speed is not polynomial due
   to the fact that the localization involves the fourth
  power of the $\WW^{1,2}$ norm of the time scheme $\{\bu^l\}_l$. 
  
  In the particular case of divergence-free finite elements, such as the Scott-Vogelius mixed elements \cite{ScoVog, Zhang} the localization
  set is $\Omega_l(M):= \big\{ \max_{j=0, \cdots, l} \|\bu^j\|_{\WW^{1,2}}^2 \leq M\big\}$. If the speed of convergence is similar  to that in the case 
  of general finite elements (with a higher power of
  $\ln(h^2 + T/N)$), the situation is different for an additive stochastic perturbation.  As it was
  already observed in \cite{Be-Mi_time}, where the existence of exponential moments for $\sup_{s\leq T} \|u(s)\|_{\WW^{1,2}}^2$
  was used, the localization by the square of $\| \bu^l\|_{\WW^{1,2}}$,   and the existence of exponential moments for the time scheme 
   provide a ``polynomial" speed of
  convergence of $\max_{1\leq l\leq N} |\bu^l - \bU^l|_{\LL^2}^2$ in terms of $h^2+ T/N$. The exponent depends on the viscosity $\nu$ and on the
  ``strength" of the noise.  If the strength of the noise converges to 0, this exponent converges to 1. 
  
  In all cases, coupling the above results with the strong convergence of $\max_{1\leq l\leq N} \big| u(\frac{lT}{N}) - \bu^l \big|_{\LL^2}^2$ from
  \cite{Be-Mi_time}, we deduce the strong speed of convergence of our   space-time Euler scheme. Note that due to the parabolic feature of the
  Navier-Sokes equations, the coupling of $h^2$ and $T/N$ is to be expected. 
  \smallskip
  
  The paper is organized as follows. In section \ref{preliminary} we define precisely our model, the stochastic perturbation,  and recall
  classical results for the solution $u$ to the stochastic 2D Navier-Stokes equations. Section 3 describes the fully implicit time Euler
  scheme; it   recalls some bounds for moments proved in \cite{BrCaPr} and \cite{CarPro}, 
  as well as the strong speed of convergence of this time scheme
  proved in \cite{Be-Mi_time} for multiplicative and additive stochastic perturbations. It also recalls some results about averages of the
  time discretization of the pressure proved in \cite{CarPro}.  In section \ref{s4}, we introduce the finite elements, 
   recall the definition of $\bU^l$ from \cite{BrCaPr, CarPro}, as well as various bounds of moments
  for $\bU^l$. Some technical estimates about the non-linear terms are proved, among which terms using some  localization on  abstract subsets of $\Omega$.
  Choosing these subsets in an appropriate way and using a discrete version of Gronwall's lemma,
       we prove the  localized $L^2(\Omega)$-convergence of the error between space-time Euler scheme 
  $\bU^l$ and the
  fully implicit time scheme $\bu^l$  in section \ref{s4bis}.      In section \ref{framework} we introduce  a general framework
  to deduce the strong speed of convergence from localized $L^2(\Omega)$ estimates and moment bounds. 
  We then apply the general results to deduce the strong speed of convergence  of 
  $\max_{1\leq l\leq N}| \bU^l - u(lT/N)|_{\LL^2}^2$
  in the case of general finite elements,  and then in that of divergence-free ones in section \ref{s5}. Finally the existence of exponential moments for
   $\max_{1\leq l\leq N} \|\bu^l\|_{\WW^{1,2}}^2$ is proven in the Appendix, both for a deterministic and random initial condition.

   As usual, throughout the paper, $C$ will denote a constant which can change from one line to the next, and we let $C(a)$ denote  a constant
   depending on some parameter $a$.

\section{Notations and preliminary results}\label{preliminary} 
 In what follows, we will  consider velocity fields that have   mean zero  
 over  $[0,L]^2$. 
 Let ${\mathbb L}^p:=[L^p_{per}(D)]^2$ (resp. ${\mathbb W}^{k,p}:=[W^{k,p}_{per}(D)]^2$)  denote the usual Lebesgue 
 and Sobolev spaces of   periodic vector-valued functions   with mean zero over   $[0,L]^2$
endowed with the norms $|\cdot |_{\LL^p}$ (resp. $\|\cdot \|_{{\mathbb W}^{k,p}}$). 
 To ease notations, we will denote by $\|\; \cdot \; \|_k$ the ${\mathbb W}^{k,p}$-norm. 
  Let 
\begin{align*}
  H:= &\{ u\in \LL^2 \; : \; {\rm div }\;  u=0 \quad \mbox {\rm weakly in }\;  D \}, \qquad
  V:=  H \cap {\mathbb W}^{1,2}; 
  \end{align*}
   these are   separable  Hilbert spaces.  The space $H$ inherits its inner product  denoted by $(\cdot,\cdot)$ and its norm from $\LL^2$.
   The norm in $V$, inherited from ${\mathbb W}^{1,2}$, is denoted by $\| \cdot \|_V$.  
   Moreover,   let  $V'$ be the dual space of $V$ 
   with respect to the Gelfand triple,  
  $\langle\cdot,\cdot\rangle$ denotes the duality between $V'$ and $V$.    Let $A=- \Delta$ with its domain  
  $\mbox{\rm Dom}(A)={\mathbb W}^{2,2}\cap H$. 
  
  Let $b:[\WW^{1,2}]^3 \to \RR$ denote the trilinear map defined by 
  \[ b(u_1,u_2,u_3):=\int_D  \big(u_1(x)\cdot \nabla u_2(x)\big)\cdot u_3(x)\, dx, \]
  which by the incompressibility condition satisfies  
  \begin{equation} \label{B}
  b(u_1,u_2, u_3)=-b(u_1,u_3,u_2), \quad b(u_1,u_2,u_2)=0,\quad  \forall u_1\in V, \; \forall u_2,u_3\in \WW^{1,2}.
  \end{equation} 
  There exists a continuous bilinear map $B:V\times V \mapsto  V'$ such that
  \[ \langle B(u_1,u_2), u_3\rangle = b(u_1,u_2,u_3), \quad \mbox{\rm for all } \; u_i\in V, \; i=1,2,3.\]
  Furthermore, the Gagliardo-Nirenberg inequality implies 
  \begin{equation} \label{interpol}
  \|u\|_{\LL^4} ^2 \leq \bar{C} \; |u|_{\LL^2} \, |\nabla u|_{\LL^2} \leq \frac{\bar{C}}{2} \|u\|_{\WW^{1,2}}^2 
  \end{equation}
  for some positive constant $\bar{C}$. 
  Recall some well-known properties of $b$, which easily follow from the H\"older and Young inequalities:
   given any $\beta >0$   we have  for $X:=H\cap \LL^4(D)$
  \begin{align}
 & |  \langle B(u_1,u_2)\, ,\,  u_3\rangle | \leq \beta \|u_3\|_V^2 + \frac{1}{4 \beta} \|u_1\|_X\, \|u_2\|_X, \label{majB-X}\\
  & | \langle B(u_1,u_1) - B(u_2,u_2)\, , \, u_1-u_2\rangle | \leq \beta \|u_1-u_2\|_V^2 + C_\beta |u_1-u_2|_{\LL^2}^2 \|u_1\|_X^4, \label{B-B}
  \end{align}
  for  $u_i\in V$, $i=1,2,3$, 
  and $\bar{C}$ is defined by \eqref{interpol}.  Finally, recall that the Poincar\'e inequality implies the existence of a constant $\tilde{C}>0$
  such that if we set $\|| u\||^2 =: |\nabla u|_{\LL^2} + |Au|_{\LL^2}$ for $u\in {\rm Dom}\, (A)$, then
  \begin{equation} \label{Poincare}
  \|u\|_V^2 = |u|_{\LL^2}^2 + |\nabla u|_{\LL^2}^2 \leq \tilde{C} \||u\||^2. 
  \end{equation}
   Finally,  recall that since the domain $D$ is the torus with periodic boundary conditions, 
  the following identity involving the Stokes operator $A$ and the bilinear term holds (see e.g. \cite{Tem} Lemma 3.1):
  \begin{equation}   \label{A-B}
  \langle B(u,u), Au \rangle =0, \quad \forall u\in \mbox{\rm Dom}(A). 
  \end{equation}

   Let $K$ be a separable Hilbert space and $Q$ be a symmetric, positive trace-classe operator on $K$. 
  Let $(W(t), t\in [0,T])$ be a $K$-valued Wiener process with covariance operator $Q$, defined on the probability
  space $(\Omega, {\mathcal F},  ({\mathcal F}_t),  \PP)$. Let $\{e_j\}_{j\geq 1}$ denote an orthonormal basis of $K$
   made of eigenfunctions of $Q$ , with eigenvalues $\{q_j\}_{j\geq 1} $. Then 
   \[ W(t)=\sum_{j=1}^\infty \sqrt{q_j} \, \beta^j(t)\, e_j, \qquad \forall t\in [0,T],\]
   where $\{ \beta_j\}_{j\geq 1}$ are independent one-dimensional Brownian motions defined on $(\Omega, {\mathcal F},  ({\mathcal F}_t),  \PP)$.
  For technical reasons, we assume that the initial condition $u_0$ belongs to $L^p(\Omega ; V)$ for some $p\in [2,\infty]$, 
   and  only consider {\it strong solutions}  in the PDE sense. 
 Given two Hilbert spaces $H_1$ and $H_2$, let   ${\mathcal L}(H_1,H_2)$ denote the set of linear operators from $H_1$ to $H_2$. 
  The diffusion coefficient $G$ satisfies the following assumption:\\  
 
 { \bf Condition (G1)} Assume that 
  $G : \LL^2 \to {\mathcal L}(K,\WW^{1,2})$ is Lipschitz-continuous with linear growth. More precisely, 
   there exist positive constants $K_i$, $i=0,1$ 
  and $L_1$  
  such that 
  \begin{align}
  \| G(u)\|_{{\mathcal L}(K,\LL^2)}^2 &\leq K_0 + K_1 |u|_{\LL^2}^2, \label{growthG_L} \quad \forall u\in \LL^2,\\ 
\| G(u)\|_{{\mathcal L}(K,\WW^{1,2})}^2 &\leq K_0 + K_1 |u|_{\mathbb{W}^{1,2}}^2, \label{growthG_W} \quad \forall u\in \WW^{1,2},\\  
 \|G(u)-G(v)\|_{{\mathcal L}(K,\LL^2)}^2 &\leq L_1 |u-v|_{\LL^2}^2, \qquad \forall u,v\in \LL^2 . 
 \label{LipG_W}
  \end{align} 
  
  We also consider the following variant of the assumptions on the diffusion coefficient $G$ in terms of divergence-free fields.\\
  {\bf Condition (G2)} 
Assume that 
  $G : H \to {\mathcal L}(K,V)$ is Lipschitz-continuous with linear growth. More precisely, 
   there exist positive constants $K_i$, $i=0,1$ 
  and $L_1$  
  such that 
  \begin{align}
  \| G(u)\|_{{\mathcal L}(K,H)}^2 &\leq K_0 + K_1 |u|_H^2, \label{growthG_H} \quad \forall u\in H, \\
\| G(u)\|_{{\mathcal L}(K,V)}^2 &\leq K_0 + K_1 |u|_V^2, \label{growthG_V} \quad \forall u\in V,\\  
 \|G(u)-G(v)\|_{{\mathcal L}(K,H)}^2 &\leq L_1 |u-v|_{H}^2, \qquad \forall u,v\in H . 
 \label{LipG_H}
  \end{align} 
 Let ${\bf P}_H$ denote the Leray projection. Note that the conditions \eqref{growthG_L} and \eqref{growthG_W} 
 imply a linear growth of the trace of ${\bf P}_HG(u)\, Q\, G^*(u)
  {\bf P}_H$ and $A^{\frac{1}{2}} {\bf P}_H G(u) Q G^*(u)  {\bf P}_H A^{\frac{1}{2}}$.
  A similar result holds under condition {\bf (G2)}.

 We define a  strong solution of \eqref{2D-NS} as follows (see Definition 2.1 in \cite{CarPro}): 
 \begin{defn}
 We say that equation \eqref{2D-NS} has a strong  solution if:
 \begin{itemize}
 \item  $u $ is an adapted $V$-valued process,
 \item $\PP$ a.s. we have $u\in C([0,T];V) \cap L^2(0,T; \mbox{\rm Dom}(A))$,
 \item  $\PP\;  \mbox{\rm a.s.}$
  \begin{align*}
  \big(u(t), \phi\big) +& \nu \int_0^t \big( \nabla u(s), \nabla \phi\big) ds + \int_0^t \big\langle [u(s) \cdot \nabla]u(s), \phi\big\rangle ds \\
& =
 \big( u_0, \phi) + \int_0^t \big( \phi ,  G(u(s)) dW(s) \big)
 \end{align*}
for every $t\in [0,T]$ and every $\phi \in V$.
  \end{itemize}
 \end{defn}
 
 As usual, by projecting  \eqref{2D-NS} on divergence free fields when dealing with the velocity, the pressure term can be disregarded 
 and the velocity
 is implicitly in the space $V$. 
 Lemma 2.1 in \cite{CarPro} (see also \cite{BeBrMi}, Theorem 4.1) shows the following:
  \begin{theorem} \label{strong_wp}
  Assume that $u_0$ is a $V$-valued, ${\mathcal F}_0$-measurable  random variable such that $\EE \big( \|u_0\|_V^{2p}\big) <\infty$ 
  for some real number $p\in [2,\infty)$.  Assume that the condition {\bf (G1)} or {\bf (G2)} is satisfied. 
  Then  there exists a unique    solution $u$ to equation \eqref{2D-NS}. 
  Furthermore,  for some positive constant $C$ we have
  \begin{equation}   \label{bound_u}
  \EE\Big( \sup_{t\in [0,T]} \|u(t)\|_V^{2p} + \int_0^T |Au(s)|_{\LL^2}^2 \big( 1+\|u(s)\|_V^{2(p-1)}\big) ds \Big) \leq C\big[ 1+ \EE (\|u_0\|_V^{2p}) \big].
  \end{equation}
  \end{theorem}

\section{ Time Euler discretization scheme}	 \label{sec_Euler}

In this section, we describe the fully implicit time Euler scheme  of the stochastic 2D Navier-Stokes equations introduced by 
E.~Carelli and A.~Prohl in \cite{CarPro}, and recall the strong convergence proved in \cite{Be-Mi_time}. 
 Fix $N\geq 1$, let $k=\frac{T}{N}$ denote the constant time mesh,  and let  $t_l=l \frac{T}{N}$, $l=0, \cdots, N,$ denote the time grid. 
  \smallskip
 
\noindent {\bf Fully implicit Euler scheme} 
{\it Let $u_0$ be a $V$-valued,  ${\mathcal F}_0$-measurable  
random variable;  set $\bu^0=u_0$. 
Fix $N\geq 1$ and for $l=1, \cdots, N,$ find pairs 
 $\big( \bu^l, \pi^l \big) \in  V\times L^2_{per}  $ such that $\PP$  a.s. for  all  
 $\phi \in \WW^{1,2}$ and $\psi \in L^2_{per}$, 
\begin{align}    \label{full-imp1}
\big( \bu^l - \bu^{l-1} , \phi \big) +  \frac{T}{N} \Big[ \nu \big( \nabla \bu^l , \nabla \phi\big) &+ \big\langle (\bu^l\cdot\nabla \bu^l), 
\phi\big\rangle \Big] - \frac{T}{N} \big( \pi^l, \mbox{\rm div } \phi)
= \big( G( \bu^{l-1})) \,  \Delta_l W , \phi\big) , \\
\big( \mbox{\rm div } \bu^l,\psi)& = 0,  \label{fu-im-div}
\end{align}
 where $\Delta_l W = W(t_l) - W(t_{l-1})$. } \\

In this section, our aim is to recall bounds for the strong error of this Euler time scheme. Since we are looking for a $V$-valued process, 
we define the scheme for the velocity 
 projected on divergence free fields and reformulate the algorithm as follows  (see \cite[Section 3]{CarPro}).
 \begin{align}    \label{full-imp1Bis}
\big( \bu^l - \bu^{l-1} , \phi \big) +&  \frac{T}{N} \Big[ \nu \big( \nabla \bu^l , \nabla \phi\big) + \big\langle (\bu^l\cdot\nabla \bu^l), 
\phi\big\rangle \Big] 
= \big( G( \bu^{l-1}) \,  \Delta_l W , \phi\big) ,  \quad \forall \phi \in V.
\end{align}

The following result proves the existence and uniqueness of the solution $\{ \bu^l, \, l=0, \cdots, N\}$ of \eqref{full-imp1Bis}; 
it provides moment estimates for this solution.
Note that here only dyadic moments are computed because an the induction argument which relates two consecutive dyadic
numbers (see step 4 of the proof of   \cite[Lemma 3.1]{BrCaPr}).  
\begin{lemma} \cite[Lemma 3.1]{CarPro} 		\label{moments_uN}
Let $u_0\in L^{2^q}(\Omega;V)$ for some integer $q\in [1,\infty)$ be ${\mathcal F}_0$-measurable,
 such that $\EE\big( \|u_0\|_V^{2^q}\big) \leq C$. Assume that $G$ satisfies
condition {\bf (G1)}. Then there exists a unique solution $\{\bu^l\}_{l=0}^N$ to \eqref{full-imp1Bis} with $\bu^0=u_0$. The random variables
$\bu^l$ are ${\mathcal F}_{t_l}$-measurable and belong to $L^2(\Omega;V)$ a.s. Furthermore,
\begin{align}
 \EE\Big( \max_{0\leq l\leq N} \|\bu^l\|_V^{2^q} + 2\nu \frac{T}{N} \sum_{l=1}^N \| \bu^l \|_V^{2^q-2}\; |A \bu^l |^2_{\LL^2}\Big)
&\leq C_1(T,q), \label{C(T,q)}\\
\EE\Big( \sum_{l=1}^N \|\bu^l - \bu^{l-1}\|_V^2 \; \| \bu^l\|_V^2\Big) & \leq C_2(T,2),    \label{C(T,2)}\\
\EE \Big[ \Big( \sum_{l=1}^N \| \bu^l- \bu^{l-1}\|_V^2 \Big)^{2^q} + \Big( \nu \frac{T}{N} \sum_{l=1}^N |A \bu^l|_{\LL^2}^2 \Big)^{2^q} \Big] & 
\leq C_3(T,q),\; q=1,2,
\label{C(T,4)}
\end{align}
where  for $i=1,2,3$, $C_i(T,q):=C_i\big(T,q, {\rm Tr}(Q), K_0, K_1, L_1, \EE(\|u_0\|_V^{2^q})\big)$ is a constant which does not depend on $N$.
\end{lemma} 

The following result about the pressure term will be used in the study of space-time discretization.
\begin{lemma}\cite[Lemma 3.2]{CarPro} Let $q\in [2,\infty)$ and  $\{\bu^l, \pi^l\}_{l=0, \cdots, N}$ be the solution to \eqref{full-imp1} and 
\eqref{fu-im-div}. \\
(i) Suppose that $G$ satisfies the growth and Lipschitz conditions {\bf (G1)}. Then
\begin{equation} 		\label{pressure-W}
\EE\Big( \frac{T}{N} \sum_{l=1}^N |\nabla \pi^l |_{\LL^2}^2\Big) \leq C(T) N.
\end{equation}
(ii) Suppose that $G$ satisfies the growth and Lipschitz conditions {\bf (G2)}. Then
\begin{equation} 		\label{pressure-V}
\EE\Big( \frac{T}{N} \sum_{l=1}^N |\nabla \pi^l |_{\LL^2}^2\Big) \leq C(T).
\end{equation}
\end{lemma}

The following theorem gives the strong rate of convergence for a multiplicative noise with linear growth. If $u_0$ is deterministic or Gaussian,
then the speed of convergence is any negative power of $\ln(N)$. 
\begin{theorem}	\cite[Theorem 4.4]{Be-Mi_time}			\label{th_Euler_lingrow}
Let $u_0$ be such that $\EE(\|u_0\|_V^{2^q})<\infty$ for some $q\geq 3$, $G$ satisfy assumption  {\bf(G1)} or {\bf(G2)}. 
Then the fully implicit scheme  
$\{ \bu^l\}_l$  solution of \eqref{full-imp1Bis}  converges in $L^2(\Omega)$ to the solution $u$ of \eqref{2D-NS}. 
More precisely, for $N$ large enough we have
\begin{align}			\label{speed_lin_growth_Euler}
\EE\Big( \max_{1\leq l \leq N} |u(t_l)- \bu^l|_{\LL^2}^2  + \frac{T}{N} \sum_{l=1}^N \big|\nabla\big( u(t_l)-\bu^l\big) \big|_{\LL^2}^2 \Big) 
\leq    C \big[\ln(N)\big]^{-(2^{q-1} - 1)}.    
\end{align}
\end{theorem}

The next result proves a better rate of convergence when the noise is additive. This is due to the existence of finite exponential moments
for $\sup_{t\in [0,T]} \|u(t)\|_V$.
\begin{theorem} \cite{Be-Mi_time} {\rm [Theorem 4.6] } \label{th_Euler_exp}
Let $u_0\in V$, $G$ satisfy assumption {\bf (G1)}  with $K_1=0$. 
Let $u$ denote the solution of
\eqref{2D-NS} and $\{ \bu^l\}_l$ be the fully implicit scheme solution of \eqref{full-imp1Bis}. Then for $N$ large enough, 
$\bar{C}$ (resp.  $\tilde{C}$) defined by \eqref{interpol} (resp. \eqref{Poincare}), 
\begin{equation}			\label{speed_add_Euler}
\EE\Big( \max_{1\leq l\leq N} |u(t_l)-\bu^l|_{\LL^2}^2 + \frac{T}{N} \sum_{l=1}^N \big| \nabla \big[ u(t_l)-\bu^l \big] \big|_{\LL^2}^2\Big)
\leq  C \Big( \frac{T}{N}\Big)^{\beta}, 
\end{equation}
where    for $\tilde{\alpha}:= \frac{\nu}{2\,  K_0 {\rm Tr} \, Q \,  \tilde{C}}$ we  have 
\begin{equation}		\label{beta0}
\beta < \beta_0:= \frac{1}{2} \left(\frac{\tilde{\alpha}_0}{\tilde{\alpha}_0 + \frac{\bar{C}^2 T}{2\nu}}\right).
\end{equation} 
\end{theorem} 
 In reference \cite{Be-Mi_time}, we had different assumptions for the splitting scheme and the Euler one. The exponential
moments were proven for the splitting scheme setting and when applying the result for the Euler scheme, we forgot to insert
the trace of $Q$.    Furthermore, using Proposition \cite[Proposition 4.3]{Be-Mi_time}, 
Theorem \ref{th_mom_exp_u}, Theorem \ref{general_strong} and Corollary \ref{general_cor},  
we can change $\tilde{\alpha}_0$ given in \cite{Be-Mi_time} to the above constant.

\section{Space-time discretization}		\label{s4}
\subsection{Description of the finite elements method}
When studying a space time discretization using finite elements,
 one needs to have a stable pairing of the velocity and the pressure which
 satisfy the discrete LBB-condition (see e.g. \cite{CarPro}, page 2469 and pages 2487-2489). Stability issues are crucial, and the pressure has
 to be discretized together with the velocity.
 
 Let ${\mathcal T}_h$ be a quasi-uniform triangulation of the domain $D\subset \RR^2$, using triangles of maximal diameter $h>0$, and set
 $\bar{D}=\cup_{K\in {\mathcal T}_h} \bar{K}$. Let $\PP_i(K):=[P_i(K)]^2$ denote the space of polynomial vector fields  on $K$
  of degree less than or equal to $i$. Given non negative integers $i,j$, we introduce finite element function spaces 
  \begin{align*}
  {\HH}_h &\,  := \{ {\bf U} \in C^0(\bar{D}) \cap \WW^{1,2}_{per}(D) : {\bf U}\in  \PP_i(K), \quad \forall K\in {\mathcal T}_h\}, \\
  L_h &\, := \{ \Pi \in L^2_{per}(D) : \Pi\in P_j(K), \quad \forall K\in {\mathcal T}_h\},
  \end{align*}
which satisfy the discrete LBB-condition 
\begin{equation}		\label{LBB}
\sup_{\bPhi \in \HH_h} \frac{(\mbox{\rm div }  \bPhi , \Pi)}{|\nabla \bPhi |_{\LL^2} } \geq C\; |\Pi|_{\LL^2}, \quad \forall \Pi \in L_h,
\end{equation}
with a constant $C>0$ independent of the mesh size $h>0$.  Here $C^0(\bar{D})$ denotes the set of continuous vector fields on $\bar{D}$. 

Define the subset $\VV_h \subset \HH_h$  of discrete divergence-free vector fields 
\begin{equation} 		\label{Vh}
\VV_h:= \{ \bPhi \in \HH_h \; : \; ({\rm div }\,  \bPhi, \Lambda)=0, \quad \forall \Lambda \in L_h\}.
\end{equation}
 Note that in general $\VV_h \not\subset V$. 


A way around this problem is to choose a space approximation such  that $\VV_h\subset V$, such as the Scott-Vogelius mixed elements 
(see \cite{ScoVog} and \cite{Zhang}). 
This particular case yields a better approximation. Indeed, on one hand the pressure will not appear in the upper estimate, and on the
other hand a different localization will provide a polynomial error in the case of an additive noise. 

Let ${\bf Q}_h^0 : \LL^2 \mapsto \VV_h$ (resp. $P^0_h:L^2_{per}(D)\mapsto L_h$) denote the orthogonal projection defined by
\begin{equation} 		\label{Q0h}
({\bf z}- {\bf Q}^0_h {\bf z} , {\bf \Phi})=0, \quad \forall {\bf \Phi} \in \VV_h, \qquad {\rm (resp. } \quad (z-P^0_hz, \Lambda)=0,
 \quad \forall \Lambda \in L_h{\rm )}.
\end{equation} 
The following estimates are standard (see e.g. \cite{HeyRan})
\begin{align}
|{\bf z} - {\bf Q}_h^0 {\bf z}|_{\LL^2} + h |\nabla ({\bf z} - {\bf Q}^0_h {\bf z})|_{\LL^2} & \leq \; C\,  h^2 \,  |A{\bf z}|_{\LL^2}, \quad 
\forall {\bf z}\in V \cap
\WW^{2,2}(D),	\label{I-Q_1}\\
|{\bf z} - {\bf Q}^0_h {\bf z}|_{\LL^2} & \leq \; C\, h \, |\nabla  {\bf z}|_{\LL^2}, \quad \forall {\bf z}\in V,   \label{I-Q_2} \\
|z-P^0_h z|_{L^2} & \leq C\, h\, |\nabla z|_{L^2}, \qquad \forall z\in W^{1,2}_{per}(D).		\label{I-P}
\end{align}

Using the Gagliardo-Nirenberg inequality \eqref{interpol}, we deduce from \eqref{I-Q_1} and \eqref{I-Q_2}  the following upper estimates for
${\bf z}\in V \cap \WW^{2,2}(D)$
\begin{equation} 		\label{I-Q_L4}
\|{\bf z}-\bQ {\bf z}\|_{\LL^4} \leq C h^{\frac{3}{2}} |A{\bf z}|_{\LL^2}, \quad {\rm and} \;  \|{\bf z}-\bQ {\bf z}\|_{\LL^4} \leq 
C h\, |A{\bf z}|_{\LL^2}^{\frac{1}{2}} \, |\nabla {\bf z}|_{\LL^2}^{\frac{1}{2}}.
\end{equation}

\subsection{Description of the space-time schemes} 
Unlike the time discretization, we will need to approximate the initial condition $u_0$ and replace it by an $\HH_h$-valued random variable $\bU^0$. 
As in \cite{CarPro}, we suppose that
\begin{equation} 		\label{U0}
\EE( |u_0 - \bU^0|_{\LL^2}^2 ) \leq C\; h^2, \qquad \EE(|\nabla \bU^0|_{\LL^2}^2) \leq C
\end{equation}
for some positive constant $C$. As it is usual in this framework, to ease notations we let $k:=\frac{T}{N}$ denote the constant time mesh. 
\\
For general finite elements satisfying the discrete LBB condition  \eqref{LBB}, one has to change the tri-linear term
$b(\bU_1, \bU_2,\bU_3 ) = ([\bU_1 \cdot \nabla] \bU_2\, , \bU_3)$ to control the nonlinear effect
 in the presence of discretely divergence-free velocity iterates,
and thus to allow stability of the scheme. Thus, we set 
\begin{equation}			\label{tildeb}
\tilde{b}(\bU_1, \bU_2, \bU_3) := ([\bU_1 \cdot \nabla ] \bU_2\, , \, \bU_3) + \frac{1}{2} \big( [{\rm div} \bU_1]\, \bU_2 \ , \bU_3\big), \quad
\forall \bU_1, \bU_2, \bU_3\in \WW^{1,2}.
\end{equation}
Note that this trilinear term is anti-symmetric with respect to the last two variables, i.e., 
\begin{equation}		\label{anti}
\tilde{b} (\bU_1, \bU_2, \bU_3) = - \tilde{b}(\bU_1, \bU_3, \bU_2), \qquad \forall \bU_1, \bU_2, \bU_3 \in \WW^{1,2}.
\end{equation}
Therefore, 
\[ 
\tilde{b}(\bU, \bPhi, \bPhi)=0, \qquad \forall \bU, \bPhi \in \WW^{1,2}.
\] 
\medskip

\noindent {\bf  Algorithm 1.}
 {\it Let $\bU^0$ be an  ${\mathcal F}_0$-measurable, $\HH_h$-valued random variable. For every $l=1, \cdots, N$,
we consider a pair of $\HH_h\times L_h$ random variables $(\bU^l, \Pi^l)$ such that for every pair $(\bPhi,\Lambda)\in \HH_h\times L_h$,
 we have a.s. }
\begin{align}		\label{algo_FE}
(\bU^l - \bU^{l-1}, \bPhi) + k \, \nu \, (\nabla \bU^l, \nabla \bPhi) + k \big( [\bU^{l-1} \cdot \nabla] \bU^l, \bPhi\big)
&+ \frac{k}{2} \big( [{\rm div} \, \bU^{l-1}]\, \bU^l, \bPhi\big) \nonumber  \\
- k ( \Pi^l, {\rm div }\,  \bPhi) &= (\bG(\bU^{l-1} ) \Delta_l W, \bPhi),\\
({\rm div }\,  \bU^l, \Lambda)&=0. \label{divU}
\end{align}

The following result, which states the existence and uniqueness of the pairs ${(\bU^l, \Pi^l)}_{l=1, \cdots, N}$ and provides moments
of the solution, 
 has been proven in \cite[Lemma 3.1]{BrCaPr}  when $\bU^0$ is deterministic (see also \cite[Lemma 4.1]{CarPro} for a random initial condition).
  Once more the exponents are dyadic
numbers because of an induction argument which enables to deduce results when doubling the exponent. 
\begin{lemma}		\label{moments_U}
Let $q\in [1, \infty)$ and $\bU^0$ be an  ${\mathcal F}_0$-measurable, $\HH_h$-valued  random variable which satisfies \eqref{U0}, and such that
$\EE\big(|\bU^0|_{\LL^2}^{2^q}\big) \leq C$ for some positive constant $C$ independent of $h>0$. Suppose that the coefficient $G$ satisfies
condition {\bf (G1)}. Then for every $l=1, \cdots, N$, there exists a unique pair $(\bU^l, \Pi^l)$ of ${\mathcal F}_{l k}$-measurable, 
 $\VV_h\times L_h$-valued  random variables which satisfy \eqref{algo_FE}-\eqref{divU}. Furthermore,
\begin{align}
&\EE\Big( \max_{0\leq l\leq N} |\bU^l|_{\LL^2}^{2^q} + \nu\, k \sum_{l=1}^N |\bU^l|_{\LL^2}^{2^{q-1}} |\nabla \bU^l|_{\LL^2}^2 \Big) \leq 
{\bf C}_1(T,q), 
   \label{U2}\\
&\EE \Big[ \Big( k \sum_{l=1}^N |\nabla \bU^l|^2_{\LL^2}\Big)^{2^{q-1}}\Big] \leq {\bf C}_2(T,q), \label{sum_grad} \\
& \EE\Big(  \sum_{l=1}^N |\bU^l-\bU^{l-1}|_{\LL^2}^2\Big) \leq {\bf C}_3(T,1), \label{U-U_L2}
\end{align}
where the constants    ${\bf C}_i(T,q):={\bf C}_i\big(T,q, {\rm Tr}\, Q, K_0, K_1, L_1, \EE(|\bU^0|_{\LL^2}^{2^q})\big)$, $i=1,2,3$
 do not depend on $N$ and $h>0$.
\end{lemma} 

We can reformulate the algorithm \eqref{algo_FE}-\eqref{divU} as follows, using divergence-free test  functions (see \cite[(4.4)]{CarPro}).\\
{\bf  Algorithm 2.}
{\it We have a.s. for $l=1, \cdots, N$ }
\begin{align}			\label{algo_special_FE}
(\bU^l - \bU^{l-1}, \bPhi) +  \nu \, k\,  (\nabla \bU^l, \nabla \bPhi) + &k \big( [\bU^{l-1} \cdot \nabla] \bU^l, \bPhi\big)
+ \frac{k}{2} \big( [{\rm div} \bU^{l-1}]\, \bU^l, \bPhi\big) \nonumber  \\
 &= (\bG(\bU^{l-1}) \Delta_l W, \bPhi), \qquad \forall \bPhi \in \VV_h .
\end{align} 
As in \cite{CarPro}, we will compare the space-time scheme $\bU^l$ and the fully implicit time scheme $\bu^l$. 
For $l=0, \cdots, N$, let $\bE^l:=\bu^l- \bU^l$. For every $l=1, \cdots, N$ and  $\bPhi\in \HH_h$, 
\begin{align*}
(\bE^l-\bE^{l-1}\, , \,  \bPhi) &+ \nu k (\nabla \bE^l \, ,\,  \nabla \bPhi) + k \tilde{b}(\bu^l, \bu^l, \bPhi) 
- k \tilde{b}( \bU^{l-1}, \bU^l, \bPhi) \\
& - k \big( \pi^l-\Pi^l\, ,\, {\rm div }\, \bPhi\big) = \big( \big[ G(\bu^{l-1}) - G(\bU^{l-1})\big]  \Delta_lW \, , \,  \bPhi\big). 
\end{align*}
Since $\bU^l\in \VV_h$ for $l=1, \cdots, N$,  we have $\bQ \bU^l=\bU^l$; in the above identity,  choose 
\[ 
\bPhi : = \bQ \bE^l = \bE^l - \big( \bu^l-\bQ \bu^l\big).
\] 
Then, since $\bQ \bE^l\in \VV_h$ and $\Pi^l\in L_h$, using \eqref{Vh} we deduce that  $\big( \Pi^l, {\rm div}\, \bQ \bE^l\big)=0$ for $l=1, \cdots, N$. 
Since $(a, a-b) = \frac{1}{2} \big( |a|_{\LL^2}^2 - |b|_{\LL^2}^2 + |a-b|_{\LL^2}^2\big)$, we deduce
\begin{align}		\label{E^m-E^(m-1)}
\frac{1}{2} &\Big[ |\bQ \bE^l|_{\LL^2}^2 - |\bQ \bE^{l-1}|_{\LL^2}^2 + |\bQ (\bE^l - \bE^{l-1}|_{\LL^2}^2 \Big] 
+ \nu k |\nabla  \bE^l|_{\LL^2}^2 
\nonumber \\
&\quad + k \tilde{b}\big(\bu^l-\bu^{l-1}, \bu^l , \bQ \bE^l\big) +k  \tilde{b}\big(\bu^{l-1}, \bu^l ,\bQ \bE^l\big) -
k  \tilde{b}\big(\bU^{l-1}, \bU^l,  
\bQ \bE^l\big) \nonumber \\
& = \nu k \big( \nabla \bE^l, \nabla(\bu^l - \bQ \bu^l) \big) + k (\pi^l, {\rm div}\, \bQ \bE^l) 
+ \Big( \big[ G(\bu^{l-1})) - G(\bU^{l-1})\big]  \Delta_l W\, , \, \bQ \bE^l\Big). 
\end{align}

\subsection{ Intermediate results}

In this section, we will 
 prove  a series of estimates that will be crucial for the convergence results  later on. 

For $l=0, \cdots, N$, let $\Omega_l\subset \Omega$ be a decreasing family of ${\mathcal F}_{t_l}$-adapted sets, that is 
\begin{equation} \label{Omegal}
\Omega_{l+1} \subset 
\Omega_l, \; l=0, \cdots, N-1, \quad {\rm and } \; \; \Omega_l \in {\mathcal F}_{t_l}, \; l=0, \cdots, N.
\end{equation} 
{\bf Part 1: Estimate of the error term $\big|\big( \nabla \bE^l\, ,\,  \nabla [\bu^l-\bQ \bu^l]\big)\big|$}\\
Using  
the Cauchy-Schwarz and  Young inequalities,  and then \eqref{I-Q_1}, we obtain for $\epsilon >0$
 \begin{align*}		
 \big|  \big( \nabla \bE^l\, ,\,  \nabla (\bu^l-\bQ \bu^l)\big)\big|\leq & \; \big| \nabla \bE^l \big|_{\LL^2} \; \big|\nabla (\bu^l - \bQ \bu^l)\big|_{\LL^2}\\
 \leq &\;  \epsilon \, \big| \nabla \bE^l \big|_{\LL^2}^2 + C( \epsilon) \, \big|\nabla (\bu^l - \bQ \bu^l)\big|^2_{\LL^2} \\
  \leq  &\; \epsilon \, \big| \nabla \bE^l \big|_{\LL^2}^2 + C( \epsilon) \, h^2 \, |A\bu^l|_{\LL^2}^2. 
 \end{align*}
This implies for $\epsilon >0$ and $m=1, \dots, N$ 
\begin{align}		\label{E1}
\EE \Big( \max_{1\leq n\leq m } &\; \nu \, k \sum_{l=1}^n 
\big| 1_{\Omega_{l-1}}\, \big( \nabla \bE^l\, ,\,  \nabla (\bu^l-\bQ \bu^l)\big)\big| \nonumber \\
&\leq \; \epsilon \nu  \EE\Big( k \sum_{l=1}^m  1_{\Omega_{l-1}}
 \, |\nabla \bE^l|_{\LL^2}^2 \Big) + C(\epsilon)\, \nu\,  k\, h^2\, \EE\Big( 
\sum_{l=1}^N \nu \, |A \bu^l|_{\LL^2}^2\Big) \nonumber \\
& \leq\;  \epsilon \nu \EE\Big( k \sum_{l=1}^m  1_{\Omega_{l-1}} \, |\nabla \bE^l|_{\LL^2}^2 \Big) + C(  \nu, \epsilon)\;  h^2\;  C_1(T,1),
\end{align}
where in the last upper estimate we have used \eqref{C(T,q)} in Lemma \ref{moments_uN} with $q=1$. 

\noindent {\bf Part 2: Estimate of the term $\tilde{b}(\bu^l-\bu^{l-1}, \, \bu^l, \, \bQ \bE^l)$}\\
Using the antisymmetry properties \eqref{B} and \eqref{anti}, the H\"older inequality
 and the inequality  $\|u\|_{\LL^4}\leq C\; | \nabla u|_{\LL^2}$ coming from the Sobolev embedding theorem, we deduce
\begin{align*}
\big| {b}(\bu^l-\bu^{l-1}, &\, \bu^l, \, \bQ \bE^l) \big| + \big| \tilde{b}(\bu^l-\bu^{l-1}, \, \bu^l, \, \bQ \bE^l) \big|  \\
& \leq 
 \; 2 \big| {b}(\bu^l-\bu^{l-1},  \bQ \bE^l, \, \bu^l ) \big| + \frac{1}{2} \big| \big( [{\rm div}\, (\bu^l - \bu^{l-1})] \; \bQ \bE^l\, , \, \bu^l\big) \big|  \\
 & \leq  \; 2 \, \| \bu^l - \bu^{l-1}\|_{\LL^4}  \; |\nabla \bQ \bE^l|_{\LL^2}\; \|\bu^l\|_{\LL^4}  + \frac{1}{2} | \nabla (\bu^l-\bu^{l-1}|_{\LL^2} 
\|\bQ \bu^l\|_{\LL^4}\; \| \bu^l\|_{\LL^4} \\
& \leq \; \epsilon\,  \nu\; |\nabla \bE^l|^2_{\LL^2} + C(\nu, \epsilon) \| \bu^l - \bu^{l-1}\|^2_V\; \|\bu^l\|^2_V.
\end{align*}
Therefore, the upper estimate \eqref{C(T,2)} in Lemma \ref{moments_uN} imply
\begin{align}		\label{b(u-u)}
\EE\Big( \max_{1\leq n\leq m} &k \sum_{l=1}^n 1_{\Omega_{l-1}}
 \big[ \big| {b}(\bu^l-\bu^{l-1}, \, \bu^l, \, \bQ \bE^l) \big| + \big| \tilde{b}(\bu^l-\bu^{l-1}, \, \bu^l, \, \bQ \bE^l) \big| \big] \Big) \nonumber \\
 &\leq \, \epsilon\,  \nu\, \EE\Big(k  \sum_{l=1}^m 1_{\Omega_{l-1}} |\nabla \bE^l|_{\LL^2}^2\Big) + C(\nu, \epsilon) \, C_2(T,2)\, k.
\end{align}

{\bf Part 3: Estimate of the term $-\tilde{b}(\bu^{l-1}, \, \bu^l, \, \bQ \bE^l) + \tilde{b}(\bU^{l-1},\, \bU^l, \, \bQ \bE^l)$}\\
\noindent For every $l=1, \cdots, N$, we have
\begin{align*}
-\tilde{b}(\bu^{l-1}, \, \bu^l, \, \bQ \bE^l) &+ \tilde{b}(\bU^{l-1},\, \bU^l, \, \bQ \bE^l)\\
&=\; -\tilde{b}(\bu^{l-1},\, \bu^l-\bU^l,\, \bQ \bE^l) - \tilde{b}(\bu^{l-1}-\bU^{l-1},\, \bU^l, \bQ\bE^l).
\end{align*}
The antisymmetry properties of $b$ (resp. $\tilde{b}$) in \eqref{B} (resp.  \eqref{anti})  implies for $l=1, \cdots, N$
\begin{align} 	\label{tildeb-tildeb}
-\tilde{b}(\bu^{l-1}, \, \bu^l, \, \bQ \bE^l) + \tilde{b}(\bU^{l-1}, \, \bU^l, \, \bQ \bE^l) = \sum_{i=1}^4 \tilde{T}_i(l), \\
- {b}(\bu^{l-1}, \, \bu^l, \, \bQ \bE^l) + {b}(\bU^{l-1}, \, \bU^l, \, \bQ \bE^l) = \sum_{i=1}^4 {T}_i(l), \label{b-b}
\end{align} 
where 
\begin{align*}
\tilde{T}_1(l)=& \;  \tilde{b}(\bu^{l-1},\, \bE^l,\, \bE^l-\bQ \bE^l), \quad \tilde{T}_2(l)= \tilde{b}(\bE^{l-1},\, \bE^l,\, \bQ\bE^l - \bE^l), \\
\tilde{T}_3(l)=& \; -\tilde{b}(\bE^{l-1},\, \bu^{l-1},\, \bQ\bE^l), \quad \tilde{T}_4(l)=\tilde{b}(\bE^{l-1},\, \bQ\bE^l,\, \bu^l-\bu^{l-1}),
\end{align*}
and the terms $T_i(l)$, $i=1, \cdots, 4$ are similar to the corresponding terms $\tilde{T}_i(l)$ replacing $\tilde{b}$ by $b$. 
\smallskip

{\it Terms $\tilde{T}_1(l)$ and $T_1(l)$.} The antisymmetry of $\tilde{b}$ and the identity
\begin{equation}	\label{E-QE}
 \bE^l-\bQ \bE^l = \bu^l-\bQ \bu^l
 \end{equation}
imply that $\tilde{T}_1(l)=-\tilde{b}(\bu^{l-1},\, \bu^l-\bQ\bu^l,\, \bE^l)$. Using the H\"older and
Gagliardo-Nirenberg inequalities,  \eqref{interpol} and \eqref{I-Q_1}, we deduce 
\begin{align*}		
\big| b(\bu^{l-1},\, \bu^l - \bQ\bu^l,  \bE^l)\big| & \leq\;  \|\bu^{l-1}\|_{\LL^4} |\nabla(\bu^l-\bQ\bu^l)|_{\LL^2} \| \bE^l\|_{\LL^4} \\ 
&\leq \; \sqrt{\bar{C}}  \, | \nabla \bE^l|_{\LL^2}^{\frac{1}{2}} |\bE^l|_{\LL^2}^{\frac{1}{2}} \|\bu^{l-1}\|_V\;  h \; | A\bu^l|_{\LL^2} \\ 
 &\leq \; \epsilon \nu |\nabla \bE^l|_{\LL^2}^2 + 
 |\bE^l|_{\LL^2}^2  + 
 C( \nu, \epsilon) \, h^2\, \|\bu^{l-1}\|_V^2 |A\bu^l|_{\LL^2}^2, 
\end{align*}
where the last upper estimate is deduced using Young's inequality. Furthermore, using once more the Gagliardo-Nirenberg inequality,
\eqref{I-Q_L4}  and then Young's inequality, we obtain
\begin{align*}
\big| \big( [{\rm div}\, \bu^{l-1}](\bu^l-\bQ\bu^l),\, \bE^l\big)\big| \leq &\; |\nabla \bu^{l-1}|_{\LL^2} \, \|\bu^l-\bQ\bu^l\|_{\LL^4} \|\bE^l\|_{\LL^4} \\
\leq & \;  \sqrt{\bar{C}}  \,  |\nabla \bE^l|_{\LL^2}^{\frac{1}{2}} |\bE^l|_{\LL^2}^{\frac{1}{2}} \|\bu^{l-1}\|_V 
\; h^{\frac{3}{2}} |A\bu^l|_{\LL^2} \\
\leq & \; \epsilon \nu |\nabla \bE^l|_{\LL^2}^2 + 
 |\bE^l|_{\LL^2}^2  + 
 C( \nu, \epsilon) \, h^3\, \|\bu^{l-1}\|_V^2 |A\bu^l|_{\LL^2}^2. 
\end{align*}
Therefore, the Cauchy-Schwarz inequality, the upper estimates \eqref{C(T,q)} with $q=2$ and \eqref{C(T,4)} with $q=1$ imply for $m=1, \cdots, N$
\begin{align}		\label{tildeT1}
\EE\Big(& \max_{1\leq n\leq m} k \sum_{l=1}^n 1_{\Omega_{l-1}} \, |\tilde{T}_1(l)|\Big) \leq  
\; 2\epsilon \nu \EE\Big( k \sum_{l=1}^m 1_{\Omega_{l-1}} |\nabla \bE^l|_{\LL^2}^2\Big)  + 2 
 \EE\Big( k \sum_{l=1}^m
1_{\Omega_{l-1}} |\bE^l|_{\LL^2}^2\Big) \nonumber \\
&  \quad  + C(\nu, \epsilon)\,  h^2 \, \Big\{ \EE\Big( \max_{0\leq l\leq N} \|\bu^l\|_V^4\Big) \Big\}^{\frac{1}{2}} 
\Big\{ \EE\Big( \nu k \sum_{l=1}^N |A\bu^l|_{\LL^2}^2 \Big)^2 \Big\}^{\frac{1}{2}}\nonumber \\
& \leq \; 2\epsilon \nu \EE\Big( k \sum_{l=1}^m 1_{\Omega_{l-1}} |\nabla \bE^l|_{\LL^2}^2 \Big)+ 2
\EE\Big( k \sum_{l=1}^m
1_{\Omega_{l-1}} |\bE^l|_{\LL^2}^2\Big)  + C(\nu,\epsilon) \, h^2. 
\end{align}
A similar computation implies
\begin{align}		\label{T1}
\EE\Big( \max_{1\leq n\leq m} k \sum_{l=1}^n 1_{\Omega_{l-1}} \, |{T}_1(l)|\Big) \leq  &\; 
\epsilon \nu \EE\Big( k \sum_{l=1}^m 1_{\Omega_{l-1}} |\nabla \bE^l|_{\LL^2}^2 \Big)+ 
\EE\Big( k \sum_{l=1}^m
1_{\Omega_{l-1}} |\bE^l|_{\LL^2}^2\Big) \nonumber \\
& \quad + C(\nu,\epsilon) \, h^2. 
\end{align}
\smallskip

{\it Terms $\tilde{T}_2(l)$ and $T_2(l)$.} Using once more \eqref{E-QE},  we can replace the difference $\bE^l-\bQ \bE^l$ by
$\bu^l-\bQ \bu^l$.
The Gagliardo-Nirenberg inequality, \eqref{I-Q_L4} and Young's inequality imply
\begin{align*}
\big| b(\bE^{l-1}, &\, \bE^l,\, \bu^l-\bQ\bu^l)\big| \leq \; \sqrt{\bar{C}} |\nabla \bE^{l-1}|_{\LL^2}^{\frac{1}{2}} |\bE^{l-1}|_{\LL^2}^{\frac{1}{2}} 
|\nabla \bE^l|_{\LL^2} \|\bu^l-\bQ\bu^l\|_{\LL^4}\\
\leq &\; \epsilon \nu |\nabla \bE^{l-1}|_{\LL^2}^2 + \epsilon \nu |\nabla \bE^l|_{\LL^2}^2 + C(\nu,\epsilon) \, h^4 \, |\bE^{l-1}|_{\LL^2}^2 
|A\bu^l|_{\LL^2}^2 |\nabla \bu^l|_{\LL^2}^2\\
\leq &\; \epsilon \nu |\nabla \bE^{l-1}|_{\LL^2}^2 + \epsilon \nu |\nabla \bE^l|_{\LL^2}^2 + C(\nu,\epsilon) \, h^4 \big( |\bu^{l-1}|_{\LL^2}^2
+ |\bU^{l-1}|_{\LL^2}^2\big) |A\bu^l|_{\LL^2}^2 |\nabla \bu^l|_{\LL^2}^2.
\end{align*}
Furthermore, the H\"older and Gagliardo-Nirenberg inequalities together with \eqref{I-Q_L4} and Young's inequality yield 
\begin{align*}
\big| \big( [{\rm div}\, \bE^{l-1}] \, &\bE^l,\, \bu^l-\bQ\bu^l\big)\big| \leq  \; \sqrt{\bar{C}}  |\nabla \bE^{l-1}|_{\LL^2} |\nabla \bE^l|_{\LL^2}^{\frac{1}{2}}
|\bE^l|_{\LL^2}^{\frac{1}{2}} \| \bu^l-\bQ\bu^l\|_{\LL^4}\\
\leq & \; \epsilon \nu |\nabla \bE^{l-1}|_{\LL^2}^2 + \epsilon \nu |\nabla \bE^l|_{\LL^2}^2 + C(\nu,\epsilon) \, h^4 \, |\bE^{l}|_{\LL^2}^2 
|A\bu^l|_{\LL^2}^2 |\nabla \bu^l|_{\LL^2}^2\\
\leq &\; \; \epsilon \nu |\nabla \bE^{l-1}|_{\LL^2}^2 + \epsilon \nu |\nabla \bE^l|_{\LL^2}^2 + C(\nu,\epsilon) \, h^4 \big( |\bu^{l}|_{\LL^2}^2
+ |\bU^{l}|_{\LL^2}^2\big) |A\bu^l|_{\LL^2}^2 |\nabla \bu^l|_{\LL^2}^2.
\end{align*} 
Since the family of sets $\Omega_l$ is decreasing, we have $1_{\Omega_{l-1}} |\nabla \bE^{l-1}|_{\LL^2}^2 \leq  
1_{\Omega_{l-2}} |\nabla \bE^{l-1}|_{\LL^2}^2 $ for $l=2, \cdots, N$. Therefore, 
 using \eqref{U0}, the Cauchy-Schwarz and Young  inequalities, \eqref{C(T,q)} and \eqref{U2} with $q=3$, and \eqref{C(T,4)} with $q=1$, we deduce
\begin{align}			\label{tildeT2}
\EE\Big(& \max_{1\leq n\leq m} k \sum_{l=1}^n 1_{\Omega_{l-1}} |\tilde{T}_2(l)|\Big) \leq \; 2\nu \epsilon  k\, \EE(|\nabla \bE^0|_{\LL^2}^2)
+ 4\nu \epsilon \EE\Big(k  \sum_{l=1}^m 1_{\Omega_{l-1}}|\nabla \bE^l|_{\LL^2}^2\Big) \nonumber  \\
&\quad + C(\nu, \epsilon) \, h^4\, \EE\Big[ \Big( \max_{0\leq l\leq N} \big\{ |\bu^l|_{\LL^2}^2+|\bU^l|_{\LL^2}^2\big\} \Big) 
\Big( \max_{1\leq l\leq N} \|\bu^l\|_V^2\Big)
\Big( \nu k \sum_{l=1}^N |A\bu^l|^2_{\LL^2}\Big) \Big] \nonumber \\
\leq &\; 2\nu \epsilon C k + 4\nu \epsilon \EE\Big(k  \sum_{l=1}^m 1_{\Omega_{l-1}}|\nabla \bE^l|_{\LL^2}^2\Big)\nonumber \\
&\; + C(\nu, \epsilon)\, h^4\,  \Big[ \Big\{ \EE\Big( \max_{0\leq l \leq N} \|\bu^l\|_V^8\Big)\Big\}^{\frac{1}{2}}  + \Big\{ \EE \Big( 
\max_{0\leq l \leq N} |\bU^l|_{\LL^2}^8\Big)\Big\}^{\frac{1}{2}} \Big] 
\Big\{ \EE\Big[ \Big( \nu k \sum_{l=1}^N |A\bu^l|_{\LL^2} \Big)^2\Big] \Big\}^{\frac{1}{2}} \nonumber \\
\leq &\;   4\nu \epsilon \EE\Big(k  \sum_{l=1}^m 1_{\Omega_{l-1}}|\nabla \bE^l|_{\LL^2}^2\Big) +  C(\nu,\epsilon) \Big[ k + \, h^4\Big].
\end{align}
A similar computation yields
\begin{equation}		\label{T2}
\EE\Big( \max_{1\leq n\leq m} k \sum_{l=1}^n 1_{\Omega_{l-1}} |{T}_2(l)|\Big) \leq \; 
2\nu \epsilon \EE\Big(k   \sum_{l=1}^m 1_{\Omega_{l-1}}|\nabla \bE^l|_{\LL^2}^2\Big) +  C(\nu,\epsilon) \big( k+ \, h^4\big).
\end{equation} 
\smallskip

{\it Terms $\tilde{T}_3(l)$ and $T_3(l)$.}  
The Gagliardo-Nirenberg inequality \eqref{interpol} implies
\[ \big| b(\bE^{l-1},\, \bu^{l-1},\, \bQ\bE^l)\big| \leq \bar{C} |\nabla \bE^{l-1}|_{\LL^2}^{\frac{1}{2}} |\bE^{l-1}|_{\LL^2}^{\frac{1}{2}} 
|\nabla \bu^{l-1}|_{\LL^2} |\nabla \bE^l|_{\LL^2}^{\frac{1}{2}} |\bE^l|_{\LL^2}^{\frac{1}{2}}.
\]
If $l=1$, Young's inequality implies 
\begin{align*}		
  \big| b(\bE^{0},\, \bu^{0},&\, \bQ\bE^1)\big| \leq \;  \bar{C}  \big( |\nabla \bE^{0}|_{\LL^2}^{\frac{1}{2}} \| \bu^{0}\|_V^{\frac{1}{2}}  
  |\bE^0|_{\LL^2}^{\frac{1}{2}}\big) \;  
 \big( \|\bu^0\|_V^{\frac{1}{2}} |\bE^1|_{\LL^2}^{\frac{1}{2}} \big) \; |\nabla \bE^1|_{\LL^2}^{\frac{1}{2}}  \\
 \leq &\; \epsilon \nu \, |\nabla \bE^1|_{\LL^2}^2 + \|\bu^0\|_V^2 |\bE^1|_{\LL^2}^2 + 
 C(\nu, \epsilon) \, \|\bu^0\|_V 
\big[  |\nabla \bu^0|_{\LL^2} + |\nabla \bU^0|_{\LL^2}\big]   |\bE^0|_{\LL^2}.
\end{align*}
Hence  H\"older's inequality, the upper estimates \eqref{U0}, \eqref{C(T,q)} for $q=2$ and \eqref{U2} for $q=2$ and imply
\begin{align} 	\label{T3-0}
k \EE\big( 1_{\Omega_0}  \big| T_3(1) \big|\big) \leq &\; k \epsilon \nu\,  \EE\big( 1_{\Omega_0}
|\nabla \bE^1|_{\LL^2}^2\big) + C k \big\{ \EE\big( \|\bu^0\|_V^4\big) \big\}^{\frac{1}{2}} \big\{ \EE \big(  |\bu^1|_{\LL^2}^2+| \bU^1|_{\LL^2}^2 \big) 
\big\}^{\frac{1}{2}} \nonumber \\
&+ k \, 
C(\nu, \epsilon)\,  \Big\{ \EE (|\bE^0|_{\LL^2}^2\big)\Big\}^{\frac{1}{2}}
\Big\{ \EE\big(\|\bu^0\|_V^4\big)\Big\}^{\frac{1}{4}} \Big\{ 2^3 \EE\big( |\bu^0|_{\LL^2}^4\big) + 2^3 \EE\big( |\bU^0|_{\LL^2}^4\big)\Big\}^{\frac{1}{4}}
\nonumber \\
\leq & \, k \epsilon \nu \, \EE\big( 1_{\Omega_0}
|\nabla \bE^1|_{\LL^2}^2\big) 
+ C(\nu,\epsilon) k (1+h),
\end{align}
for some constant $C(\nu,\epsilon)$ depending on $\bar{C}$, $\epsilon$, $\nu$, $C_1(T,2)$ and ${\bf C}_1(T,2)$. 

For $l=2, \cdots, N$,  using \eqref{interpol}, the H\"older and Young inequalities we obtain for $\epsilon, \epsilon_1  >0$
\begin{align*}
\big| b(&\bE^{l-1},\, \bu^{l-1},\, \bQ\bE^l)\big| \leq \,   \bar{C}  |\nabla \bE^{l-1}|_{\LL^2}^{\frac{1}{2}} |\nabla \bE^l|_{\LL^2}^{\frac{1}{2}} 
\Big( \|\bu^{l-1}\|_V^{\frac{1}{2}} |\bE^l|_{\LL^2}^{\frac{1}{2}}\Big) \Big(  \|\bu^{l-2}\|_V^{\frac{1}{2}} |\bE^{l-1}|_{\LL^2}^{\frac{1}{2}}\Big)  \\
&\quad   +  \bar{C} \,  |\nabla \bE^{l-1}|_{\LL^2}^{\frac{1}{2}} |\nabla \bE^l|_{\LL^2}^{\frac{1}{2}} 
\Big( \|\bu^{l-1}\|_V^{\frac{1}{2}} |\bE^l|_{\LL^2}^{\frac{1}{2}}\Big) \Big(  \|\bu^{l-1}- \bu^{l-2}\|_V^{\frac{1}{2}} |\bE^{l-1}|_{\LL^2}^{\frac{1}{2}}\Big) \\
\leq & \; \epsilon_1 \nu |\nabla \bE^{l-1}|_{\LL^2}^2 + \epsilon_1 \nu |\nabla \bE^{l}|_{\LL^2}^2 
+ \frac{  \bar{C}^2  }{2^4 \epsilon_1 \nu} \|\bu^{l-1}\|_V^2 
|\bE^l|_{\LL^2}^2 + \frac{ \bar{C}^2 }{2^4 \epsilon_1 \nu} \|\bu^{l-2}\|_V^2  |\bE^{l-1}|_{\LL^2}^2 \\
&\quad +  \epsilon \nu |\nabla \bE^{l-1}|_{\LL^2}^2 + \epsilon \nu |\nabla \bE^{l}|_{\LL^2}^2 +  \|\bu^{l-1}\|_V^2 
|\bE^l|_{\LL^2}^2 +  \frac{ \bar{C}^4  }{2^8\epsilon^2 \nu^2} \|\bu^{l-1}- \bu^{l-2}\|_V^2   |\bE^{l-1}|_{\LL^2}^2 . 
\end{align*} 
We next use the upper estimates \eqref{C(T,q)} and \eqref{U2} with $q=2$, and \eqref{C(T,4)} with $q=1$ and the inequality
$1_{\Omega_{l-1}}  \leq 1_{\Omega_{l-2}} $.  This yields for $m=2, \cdots, N$
\begin{align}			\label{T3-2-N}
\EE&\Big( \max_{2\leq n\leq m} k \sum_{l=2}^n  1_{\Omega_{l-1}}  |T_3(l) | \Big) \leq  2(\epsilon_1 + \epsilon)\,  \nu  \EE\Big( k 
\sum_{l=1}^m 1_{\Omega_{l-1}}
|\nabla \bE^l|_{\LL^2}^2  \Big) \nonumber \\
&\quad +  \Big[ 1+ \frac{2 \,  \bar{C}^2 }{2^4\epsilon_1 \nu} \Big]  \EE\Big( k \sum_{l=1}^m 1_{\Omega_{l-1}} \|\bu^{l-1}\|_V^2
|\bE^l|_{\LL^2}^2 \Big) \nonumber \\
&\;  + C(\epsilon, \nu)k   \Big\{ \EE\Big( \max_{1\leq l\leq N}  |\bu^l|_{\LL^2}^4  +  
 \max_{1\leq l\leq N} |\bU^l|_{\LL^2}^4  \Big) \Big\}^{\frac{1}{2}} \Big\{ \EE \Big[ \Big( \sum_{l=1}^N \| \bu^l-\bu^{l-1}\|_V^2\Big)^2 \Big]
\Big\}^{\frac{1}{2}}\nonumber \\
&\leq 2( \epsilon_1+\epsilon)\,  \nu  \EE\Big( k \sum_{l=1}^m 1_{\Omega_{l-1}}
|\nabla \bE^l|_{\LL^2}^2  \Big) +  \Big[1+  \frac{  \bar{C}^2  }{2^3\epsilon_1 \nu}\Big]  
\EE\Big(  k\sum_{l=1}^m 1_{\Omega_{l-1}} \|\bu^{l-1}\|_V^2
|\bE^l|_{\LL^2}^2 \Big) \nonumber \\
&\;
 + C(\nu,\epsilon)\, k.
\end{align}
The upper estimates \eqref{T3-0} and \eqref{T3-2-N} imply for $\epsilon, \epsilon_1 >0$ and $h\in(0,1]$ 
\begin{align}		\label{T3}
\EE\Big( k\max_{1\leq n\leq m}  &\sum_{l=1}^n  1_{\Omega_{l-1}}  |T_3(l) | \Big) \leq 
(2 \epsilon_1+3\epsilon)  \, \nu \,    \EE\Big( k  \sum_{l=1}^m 1_{\Omega_{l-1}}
|\nabla \bE^l|_{\LL^2}^2  \Big) \nonumber  \\
&+   \Big[1+  \frac{\bar{C}^2}{2^3\epsilon_1 \nu}\Big]    \EE\Big( k \sum_{l=1}^m 1_{\Omega_{l-1}} \|\bu^{l-1}\|_V^2
|\bE^l|_{\LL^2}^2 \Big) + C(\nu,\epsilon) \, k.
\end{align}
Furthermore, the Gagliardo-Nirenberg  \eqref{interpol} and Young inequalities imply  for $\epsilon_2 >0$ 
\begin{align*}
\big| \big( [{\rm div}\, \bE^{l-1}]\, \bu^{l-1},\, \bQ\bE^l\big)\big| \leq &\;  \frac{ \bar{C}}{\sqrt{2}  } |\nabla \bE^{l-1}|_{\LL^2}
\| \bu^{l-1}\|_V    |\nabla \bE^{l}|_{\LL^2}^{\frac{1}{2}}  |\bE^{l}|_{\LL^2}^{\frac{1}{2}}\\
  \leq & \; \epsilon_2 \nu |\nabla \bE^{l-1}|_{\LL^2}^2 + \epsilon_2 \nu |\nabla \bE^l|_{\LL^2}^2 +  
  \frac{\bar{C}^4}{2^{8} (\epsilon_2 \nu)^3}  \|\bu^{l-1}\|_V^4
  |\bE^l|_{\LL^2}^2. 
\end{align*}
Using once more the inequality $1_{\Omega_{l-1}} |\nabla \bE^{l-1}|_{\LL^2}^2 \leq 1_{\Omega_{l-2}} |\nabla \bE^{l-1}|_{\LL^2}^2 $, we deduce for
any $m=1, \cdots, N$
\begin{align}		\label{divT3}
\EE\Big( \max_{1\leq n\leq m} k \sum_{l=1}^n  1_{\Omega_{l-1}} \big| \big( [{\rm div}\, &\bE^{l-1}]\, \bu^{l-1},\, \bQ\bE^l\big)\big| \Big)  \leq 
2\epsilon_2 \nu \EE\Big(  k \sum_{l=1}^m 1_{\Omega_{l-1}} |\nabla \bE^l|_{\LL^2}^2\Big) \nonumber \\
&+   \frac{\bar{C}^4}{2^{8} (\epsilon_2 \nu)^3}  k \EE\Big( \sum_{l=1}^m 1_{\Omega_{l-1}} \|\bu^{k-1}\|_V^4  |\bE^l|_{\LL^2}^2\Big).
\end{align}
The upper estimates \eqref{T3} and \eqref{divT3} imply for $h\in (0,1]$, $\epsilon, \epsilon_1, \epsilon_2>0$, 
\begin{align}		\label{tildeT3}
\EE\Big( \max_{1\leq n\leq m} &k \sum_{l=1}^n 1_{\Omega_{l-1}} |\tilde{T}_3(l)| \Big) \leq  \; (2\epsilon_1+2\epsilon_2+3\epsilon)\, 
 \nu \, \EE\Big(k  \sum_{l=1}^m  1_{\Omega_{l-1}} |\nabla \bE^l|_{\LL^2}^2\Big) + Ck \nonumber \\
&+  \EE\Big\{ k \sum_{l=1}^m 1_{\Omega_{l-1}} \Big[  \Big( 1+\frac{\bar{C}^2}{2^3 \epsilon_1 \nu} \Big)  \|\bu^{l-1}\|_V^2
+  \frac{\bar{C}^4}{2^{8} (\epsilon_2 \nu)^3}  \|\bu^{l-1}\|_V^4 \Big] |\bE^l|_{\LL^2}^2\Big\}.
\end{align}
\smallskip

{\it Terms $\tilde{T}_4(l)$ and $T_4(l)$.}  The Gagliardo-Nirenberg and Young inequalities imply
\begin{align*}
\big| b& \big( \bE^{l-1},\, \bQ\bE^l,\, \bu^l-\bu^{l-1} \big)\big| \leq \frac{\bar{C}}{\sqrt{2}}
  \| \nabla \bE^{l-1}\|_{\LL^2}^{\frac{1}{2}} |\bE^{l-1}||_{\LL^2}^{\frac{1}{2}}
 |\nabla \bQ\bE^l|_{\LL^2} \| \bu^l-\bu^{l-1}\|_V \\
\leq &\;  \epsilon \nu |\nabla \bE^{l-1}|_{\LL^2}^2 + \epsilon \nu |\nabla \bE^l|_{\LL^2}^2 
+ C(\epsilon, \nu)   |\bE^{l-1}|_{\LL^2}^2  \| \bu^l-\bu^{l-1}\|_V^4 \\ 
\leq &\; \epsilon \nu\big(  |\nabla \bE^{l-1}|_{\LL^2}^2 +  |\nabla \bE^l|_{\LL^2}^2 \big)
+ C(\epsilon,\nu) \big( |\bu^{l-1}|_{\LL^2}^2 + |\bU^{l-1}|_{\LL^2}^2\big) \big( \| \bu^{l}\|_V^2 + \| \bu^{l-1}\|_V^2\big) \| \bu^l-\bu^{l-1}\|_V^2.
\end{align*}
For $l=1$ the inequality  \eqref{U0} imples 
\[ k \epsilon \nu  \EE \big(|\nabla \bE^0|_{\LL^2}^2 \big) \leq C  k \nu \epsilon.\]
Furthermore, the Cauchy-Schwarz inequality together with the upper estimates \eqref{C(T,q)},  \eqref{U2}  for $q=3$,  \eqref{C(T,4)} 
 for $q=1$,    
and  the inequality $1_{\Omega_{l-1}}  \leq 1_{\Omega_{l-2}}  $ imply
\begin{align}		\label{T4}
\EE\Big(k  & \max_{1\leq n\leq m} \sum_{l=1}^n 1_{\Omega_{l-1}}  |T_4(l)|\Big) \leq \; 
2\epsilon \nu \EE\Big(k  \sum_{l=1}^m 1_{\Omega_{l-1}} |\nabla \bE^l|_{\LL^2}^2 \Big) 
+ C(\nu,\epsilon) k  \nonumber \\
 & \quad  + C(\epsilon, \nu) k 
 \Big\{ \EE\Big( \Big| \sum_{l=1}^N \|\bu^l-\bu^{l-1}\|_V^2\Big|^2\Big) \Big\}^{\frac{1}{2}} 
 \Big\{ \EE\Big( \max_{1\leq l\leq N} \|\bu^l\|_V^8 + \max_{1\leq l\leq m} |\bU^l|_{\LL^2}^8  \Big\}^{\frac{1}{2}}\nonumber \\
 \leq &\;  2\epsilon \nu \EE\Big(k  \sum_{l=1}^m 1_{\Omega_{l-1}} |\nabla \bE^l|_{\LL^2}^2 \Big) + C(\nu,\epsilon) k.
\end{align}
On the other hand, the H\"older, Gagliardo-Nirenberg and Young inequalities imply
\begin{align*}
\big| \big( [{\rm div} \, \bE^{l-1}]\, &\bQ\bE^l, \, \bu^l-\bu^{l-1}\big)\big| \leq |\nabla \bE^{l-1}|_{\LL^2} \|\bE^l\|_{\LL^4} \|\bu^l-\bu^{l-1}\|_{\LL^4}\\
\leq & \; \bar{C} \,   |\nabla \bE^{l-1}|_{\LL^2} |\nabla \bE^l|_{\LL^2}^{\frac{1}{2}} \Big( |\bE^l|_{\LL^2}^{\frac{1}{2}} 
|\nabla (\bu^l-\bu^{l-1})|_{\LL^2}^{\frac{1}{2}} \max_{0\leq l\leq N} \| \bu^l\|_V^{\frac{1}{2}}\Big) \\
\leq & \; \epsilon \nu |\nabla \bE^{l-1}|_{\LL^2}^2 + \epsilon \nu |\nabla \bE^l|_{\LL^2}^2 + C(\nu, \epsilon) \Big[ \max_{0\leq l\leq N} \|\bu^l\|_V^4
+ \max_{0\leq l\leq N} |\bU^l|_{\LL^2}^4\Big] \| \bu^l-\bu^{l-1}\|_V^2. 
\end{align*}
A similar argument as for the upper estimate of $T_4(l)$, based on \eqref{U0},  \eqref{C(T,q)} and \eqref{U2} with $q=3$,  \eqref{C(T,4)} with $q=1$, 
and $1_{\Omega_{l-1}}\leq 1_{\Omega_l}$  yield
\begin{align}		\label{divT4}
\EE\Big( &\max_{1\leq n\leq m} k \sum_{l=1}^n 1_{\Omega_{l-1}} \big| \big( [{\rm div} \, \bE^{l-1}]\, \bQ\bE^l, \, \bu^l-\bu^{l-1}\big)\big| \Big) \nonumber \\
\leq & \; 2\epsilon \nu  \EE\Big( k \sum_{l=1}^m 1_{\Omega_{l-1}} |\nabla \bE^l|_{\LL^2}^2\Big) + \epsilon \nu C k 
 \EE\big(|\nabla \bE^0|_{\LL^2}^2\big)\nonumber \\
 & \; + C(\nu, \epsilon) k \Big[ \Big\{ \EE\big( \max_{1\leq l\leq N} \|\bu^l\|_V^8\Big) \Big\}^{\frac{1}{2}} +
 \Big\{ \EE\big( \max_{1\leq l\leq N} |\bU^l|_{\LL^2}^8\Big) \Big\}^{\frac{1}{2}}\Big] \Big\{ \Big| \EE\Big(  \sum_{l=1}^N |\bu^l-\bu^{l-1}\|_V^2\Big) \Big|^2
 \Big\}^{\frac{1}{2}} \nonumber \\
 \leq & \;  2\epsilon \nu  \EE\Big( k \sum_{l=1}^m 1_{\Omega_{l-1}} |\nabla \bE^l|_{\LL^2}^2\Big) + C(\nu, \epsilon)  k.
\end{align}
The upper estimates \eqref{T4} and \eqref{divT4} imply
\begin{align}		\label{tildeT4}
\EE\Big( \! \max_{1\leq n\leq m} \sum_{l=1}^m 1_{\Omega_{l-1}}|\tilde{T}_4(l)|\! \Big) \leq &\; 
  4\epsilon \nu \EE\Big(k  \sum_{l=1}^m 1_{\Omega_{l-1}} |\nabla \bE^l|_{\LL^2}^2\! \Big)  + C(\nu,\epsilon) k. 
\end{align}
The upper estimates \eqref{tildeT1}, \eqref{tildeT2}, \eqref{tildeT3} and \eqref{tildeT4} imply that for any $m=1, \cdots, N$,
\begin{align}		\label{tildeb-tildeb}
\EE\Big(&  \max_{1\leq n\leq m} k\sum_{l=1}^n 1_{\Omega_{l-1}} \big| \tilde{b}(\bu^{l-1},\, \bu^l,\, \bQ\bE^l) 
- \tilde{b}(\bU^{l-1},\, \bU^l,\, \bQ\bE^l)\big| \Big) \nonumber \\
 \leq &  \; C(\nu,\epsilon)\Big( k +  h^2+h^4\Big)  
 +  (2\epsilon_1+2\epsilon_2+13\epsilon) \, \nu \,  \EE\Big(k  \sum_{l=1}^m 1_{\Omega_{l-1}} |\nabla \bE^l|_{\LL^2}^2 \Big) \nonumber \\
&\; + \EE\Big[ k \sum_{l=1}^m 1_{\Omega_{l-1}}   \Big\{ 2+   \Big( 1+\frac{\bar{C}^2}{2^3 \epsilon_1 \nu} \Big)  \|\bu^{l-1}\|_V^2
+ \frac{\bar{C}^4}{2^{8} \epsilon_2^3 \nu^3}  \|\bu^{l-1}\|_V^4 \Big\}  |\bE^l|_{\LL^2}^2\Big].
\end{align}
Similarly, the upper estimates \eqref{T1}, \eqref{T2}, \eqref{T4} and \eqref{T4} imply that for $m=1, \cdots, N$,
\begin{align}		\label{b-b}
\EE\Big( \max_{1\leq n\leq m} &k  \sum_{l=1}^n 1_{\Omega_{l-1}} \big| {b}(\bu^{l-1},\, \bu^l,\, \bQ\bE^l) 
- {b}(\bU^{l-1},\, \bU^l,\, \bQ\bE^l)| \Big) \nonumber \\
 \leq &  \; C(\nu,\epsilon)\Big( k +  h^2+h^4\Big)  
 +  (2 \epsilon_1+8\epsilon) \,  \nu \,  \EE\Big(k  \sum_{l=1}^m 1_{\Omega_{l-1}} |\nabla \bE^l|_{\LL^2}^2 \Big) \nonumber \\
&\; + \EE\Big[ k \sum_{l=1}^m 1_{\Omega_{l-1}}   \Big\{ 1+   \Big( 1+\frac{\tilde{C}^2}{2^3 \epsilon_1 \nu} \Big) \|\bu^{l-1}\|_V^2
 \Big\}  |\bE^l|_{\LL^2}^2\Big].
\end{align}

\noindent {\bf Part 4: Estimate of $(\pi^l, {\rm div}\, \bE^l)$} \\
Since $\bQ \bE^l\in \VV_h$ and $P^0_h \pi^l\in L_h$, using \eqref{Vh} we deduce 
$(\pi^l , {\rm div}\, \bQ \bE^l) = (\pi^l- P^0_l \pi^l, {\rm div}\, \bQ \bE^l)$.
Therefore, the Cauchy-Schwarz and Young inequalities  coupled with  \eqref{I-P} 
 imply
\begin{align}		\label{error-pressure-W}
 \EE\Big(k  \sum_{l=1}^m  &1_{\Omega_{l-1}} \big| \big( \pi^l, \, {\rm div}\, \bQ \bE^l\big)\big|\Big) \leq  \; 
  \EE\Big( k \sum_{l=1}^m 1_{\Omega_{l-1}} \big[ \epsilon \, \nu\, |\nabla \bE^l|_{\LL^2}^2 
 + \frac{1}{4\epsilon \nu}  |\pi^l - P^0_h \pi^l|_{L^2}^2\big] \Big) \nonumber \\
 \leq &\;  \epsilon\, \nu\,  \EE\Big( k \sum_{l=1}^m 1_{\Omega_{l-1}} |\nabla \bE^l|_{\LL^2}^2 \Big) + 
 C(\nu, \epsilon )  h^2 \EE\Big( k \sum_{l=1}^m  1_{\Omega_{l-1}} |\nabla \pi^l|_{L^2}^2\Big). 
\end{align}

\section{Convergence in probability}\label{s4bis}  
  This section contains  results  on strong convergence  localized on a subset of $\Omega$ due to the non linear term  
  for both algorithms 1 and 2.  
 This is classical when dealing with a non linear term. However, 
  unlike \cite{CarPro}, our bound in the localization only involves the time discretization $\{\bu^l\}_l$.  In the general case, it will require an upper bound of  the fourth
power of its $V$ norm. Arguments similar to that used in \cite{CarPro} or \cite{BeBrMi}, we could deduce a rate of convergence in probability. We do not provide
this rate since our paper focuses on $L^2(\Omega)$ rates of convergence.

The following proposition is one of the main results of this section. 
\begin{prop} 		\label{loc-conv}
Let $G$ satisfy the growth and Lipschitz conditions {\bf (G1)}. Let $u_0\in L^8(\Omega; V)$ and $\bU^0\in L^8(\Omega; \HH_h)$ be 
${\mathcal F}_0$-measurable random variables such that \eqref{U0} holds. For every $M>0$ and $l=0, \cdots, N$, set
\begin{equation}		\label{tildeOmega}
\tilde{\Omega}_l(M):= \Big\{ \omega \in \Omega \, : \, \max_{0\leq j\leq l} \|\bu^j\|_V^4 \leq M\Big\}.
\end{equation}
There exists $C_0>0$ such that for $k=\frac{T}{N}>0$ small enough to have $kM\leq C_0$, we have for any $\delta >0$ 
and $m=1, \cdots, N$ 
 \begin{equation}		\label{loc_mom_gene}
\EE\Big( 1_{\tilde{\Omega}_{m-1}(M)} \max_{1\leq n\leq m}  |\bE^n|_{\LL^2}^2 \Big)  \leq  
C(\nu)\, e^{\tilde{C}_1(M) T} 
 \Big[ k+ h^2 
+  h^2\; \EE\Big(k  \sum_{l=1}^m   
\big| \nabla \pi^l\big|_{L^2}^2\Big)\Big],
\end{equation}
for some constant 
\begin{align} 	\label{C1(M)gene}
\tilde{C}_1(M) & = (1+\delta) \frac{\bar{C}^4}{2^4 \nu^3} M 
+ C (\nu, L_1 \mbox{\rm Tr}\, Q, \delta)\nonumber \\
& \equiv (1+\delta) \frac{\bar{C}^4}{2^4 \nu^3} M  \quad   \mbox{\rm for ``large"}\; M.  
\end{align} 
Furthermore, 
\begin{equation}		\label{loc_mom_gene_nabla}
\EE\Big( \! k \!\sum_{l=1}^m 1_{\tilde{\Omega}_{l-1}(M)} |\nabla \bE^l|_{\LL^2}^2 \! \Big) \leq
C(\nu) \tilde{C}_1(M) T  e^{\tilde{C}_1(M)T} 
\Big[ k+ h^2 +  h^2 \EE\Big(\! k  \!\sum_{l=1}^m 
\big| \nabla \pi^l\big|_{L^2}^2 \Big)\Big].
\end{equation} 
 \end{prop}
\begin{proof}
Note that the definition of $\tilde{\Omega}_l(M)$ clearly implies that for fixed $M>0$, $l=0, \cdots, N-1$, $\tilde{\Omega}_{l+1}(M) \subset 
\tilde{\Omega}_l(M)$. Hence, as proved in \cite{CarPro} (see also \cite{Be-Mi_time}),  we have 
\begin{align}  \label{maj_E-E}
\sum_{l=1}^n& 1_{\tilde{\Omega}_{l-1}(M) } \Big( |\bQ \bE^l |_{\LL^2}^2
- |\bQ \bE^{l-1}|_{\LL^2}^2 \Big)  =  
 \Big( 1_{\tilde{\Omega}_{n-1}(M)} |\bQ \bE^n |_{\LL^2}^2  \nonumber\\
&\quad \qquad+ \sum_{l=2}^n  \big( 1_{\tilde{\Omega}_{l-2}(M)} - 1_{\tilde{\Omega}_{l-1}(M)}\big) |\bQ \bE^{l-1}|_{\LL^2}^2 \Big)
- 1_{\tilde{\Omega}_0(M)} |\bQ \bE^0|_{\LL^2}^2 \nonumber \\
  \qquad \geq &   
    1_{\tilde{\Omega}_{n-1}(M)} |\bQ \bE^n|_{\LL^2}^2 - 1_{\tilde{\Omega}_0(M)} |\bQ \bE^0|_{\LL^2}^2, \qquad \forall n=1, \cdots, N.
\end{align} 
This inequality and  \eqref{E^m-E^(m-1)} imply for $n=1, \cdots, N$ 
\begin{align}		\label{Ediff}
\frac{1}{2}  1_{\tilde{\Omega}_{n-1}(M)} |\bQ\bE^n|_{\LL^2}^2  &+ \frac{1}{2} \sum_{l=1}^n 1_{\tilde{\Omega}_{l-1}(M) } 
\big|\bQ(\bE^l - \bE^{l-1})\big|_{\LL^2}^2 + \nu  k \sum_{l=1}^m 1_{\tilde{\Omega}_{l-1}(M)} |\nabla \bE^l|_{\LL^2}^2  \nonumber \\
& 
\leq 
\sum_{i=1}^3  \tilde{R}_i(n,k,h,M) ,
\end{align}
where
\begin{align} \label{def_R}
\tilde{R}_1(n,k,h,M)& = \;   1_{\tilde{\Omega}_0(M) } |\bQ \bE^0|_{\LL^2}^2   +
 k \sum_{l=1}^n 1_{\tilde{\Omega}_{l-1}(M)} \big| \tilde{b}(\bu^l-\bu^{l-1},\, \bu^l, \, \bQ\bE^l)\big| \nonumber \\
&+ k \sum_{l=1}^n 1_{\tilde{\Omega}_{l-1}(M)} \big| \tilde{b}(\bu^{l-1} ,\, \bu^{l},\, \bQ\bE^l) - \tilde{b}(\bU^{l-1}, \, \bU^l , \, \bQ\bE^l)
\big| \nonumber  \\
&+ k \sum_{l=1}^n 1_{\tilde{\Omega}_{l-1}(M)} \big| \big(\nabla \bE^l, \, \nabla[\bu^l-\bQ\bu^l]\big) \big| 
+ k \sum_{l=1}^n 1_{\tilde{\Omega}_{l-1}(M)} |( \pi^l,\, {\rm div}\,\bQ \bE^l)|, \nonumber \\
\tilde{R}_2(n,k,h,M)&=\;   \sum_{l=1}^n1_{\tilde{\Omega}_{l-1}(M)} \Big( \big[ G(\bu^{l-1} ) - G(\bU^{l-1})\big] 
\Delta_lW, \, \bQ\bE^{l-1}\Big) , \nonumber \\
\tilde{R}_3(n,k,h,M) & =\;  \sum_{l=1}^n1_{\tilde{\Omega}_{l-1}(M)} \Big( \big[ G(\bu^{l-1} ) - G(\bU^{l-1})\big] 
\Delta_lW, \, \bQ(\bE^l - \bE^{l-1})\Big).
\end{align}
The upper estimates \eqref{U0}, \eqref{b(u-u)}, \eqref{tildeb-tildeb}, \eqref{E1} and \eqref{error-pressure-W}  imply for any 
$\epsilon_1, \epsilon_2, \epsilon >0$ and
small $h>0$
\begin{align}		\label{R1-first}
E\big(\tilde{R}_1&(n,k,h,M\big) )\leq \; C(\nu,\epsilon) \big( k + h^2\big) 
+   C(\nu, \epsilon )\,  h^2 \, \EE\Big( k \sum_{l=1}^n 1_{\tilde{\Omega}_{l-1}(M)} |\nabla \pi^l|_{L^2}^2\Big) \nonumber \\
&\; +   (2 \epsilon_1+2\epsilon_2+ 16 \epsilon) \, \nu\,  \EE\Big(k  \sum_{l=1}^n 1_{\tilde{\Omega}_{l-1}(M)}
|\nabla \bE^l|_{\LL^2}^2\Big) \nonumber \\
&  \; + \EE\Big( k \sum_{l=1}^n 1_{\tilde{\Omega}_{l-1}(M)} \Big[ 2+\Big( 1+\frac{\bar{C}^2}{2^3\epsilon_1 \nu}\Big)
\sqrt{M} + \frac{\bar{C}^4}{2^{8}\epsilon_2^3 \nu^3} M\Big] |\bE^l|_{\LL^2}^2\Big) .
\end{align}
 Since $1_{\tilde{\Omega}_{l-1}(M)}$, 
$\bu^{l-1}$, $\bU^{l-1}$ and $\bE^{l-1}$ are independent of $\Delta_l W$, we have
$\EE\big( \tilde{R}_2(n,k,h,M)\big)=0$. 

Using the Lipschitz property of $G$ \eqref{LipG_W}, the Cauchy-Schwarz and Young inequalities,  
we obtain
\begin{align} 		\label{upper_M3}
\EE\big(&  \tilde{R}_3(n,k,h,M)  \big) \leq  \EE\Big[ \sum_{l=1}^n  1_{\tilde{\Omega}_{l-1}(M)}  
\big| \big( G(\bu^{l-1}) - G(\bU^{l-1}) \big) \Delta_lW \big|_{\LL^2}   |\bQ (\bE^l-\bE^{l-1})|_{\LL^2}\Big] \nonumber \\
\leq & \; \sum_{l=1}^n \Big\{ \EE\Big( 1_{\tilde{\Omega}_{l-1}(M)}|\bQ (\bE^l-\bE^{l-1})|^2_{\LL^2}\Big)\Big\}^{\frac{1}{2}} \nonumber \\
&\qquad \times 
\Big\{ \EE \Big( 1_{\tilde{\Omega}_{l-1}(M)} 
\| G(\bu^{l-1})-G(\bU^{l-1})\|_{{\mathcal L}(K,\WW^{1,2})}^2 k {\rm Tr}\, Q \Big) \Big\}^{\frac{1}{2}} \nonumber \\
\leq & \frac{1}{2} \EE\Big(\! \sum_{l=1}^n \! 1_{\tilde{\Omega}_{l-1}(M)}  |\bQ(\bE^l-\bE^{l-1})|_{\LL^2}^2\! \Big) + 
\frac{L_1 \, {\rm Tr}\, Q}{2} \EE\Big(\! k \sum_{l=1}^{n-1}\! 1_{\tilde{\Omega}_{l-1}(M)}  |\bE^l|_{\LL^2}^2 \!\Big) + C  h^2 k,
\end{align}
where in the last upper estimate we have used the inequalities $1_{\tilde{\Omega}_{l-1}(M)}
\leq 1_{\tilde{\Omega}_{l-2}(M)}$ for $l=2, \cdots, N$, and \eqref{U0} for $l=0$. 

 Fix $\lambda \in (0, \frac{1}{2})$,  let $\epsilon_1=\frac{\lambda}{2}$, $\epsilon_2=\frac{1}{2}(1-2\lambda)$ and $\epsilon = \frac{\lambda}{2^5}$.
Then $2\epsilon_1+2\epsilon_2+16\epsilon = 1-\frac{\lambda}{2}$.  The upper estimates \eqref{Ediff} -- \eqref{upper_M3}
yield
\begin{align} 		\label{pre-Gronwall-1}
 \EE\Big( & 1_{\tilde{\Omega}_{n-1}(M)} |\bQ\bE^n|_{\LL^2}^2 \!\Big) 
+  \lambda \nu \EE\Big( \!k \sum_{l=1}^n 
\! 1_{\tilde{\Omega}_{l-1}(M)} |\nabla \bE^l|_{\LL^2}^2\!\Big) \nonumber  \\
 \leq & \;  C(\nu) \big( k+ h^2\big) 
+ 2 h^2  \EE\Big(k  \sum_{l=1}^n 1_{\tilde{\Omega}_{l-1}(M)} | \nabla \pi^l  |_{\LL^2}^2 \Big)
+  L_1 {\rm Tr}\, Q\;  \EE\Big( k \sum_{l=1}^{n-1} 1_{\tilde{\Omega}_{l-1}(M)} |\bE^l|_{\LL^2}^2\Big)  \nonumber \\
&\; +   2 \,  \EE\Big( k  \sum_{l=1}^n  1_{\tilde{\Omega}_{l-1}(M)} \Big[ 2+  C(\nu,\bar{C},  \lambda)\; 
\sqrt{M} +  \tilde{\alpha}(\nu,\bar{C}, \lambda) 
M\Big] |\bE^l|_{\LL^2}^2\Big), 
\end{align}
 where $\tilde{\alpha}(\nu,\bar{C}, \lambda):= \frac{\bar{C}^4}{ 2^5 \, (1-2\lambda)^3\, \nu^3} $.
Using \eqref{I-Q_2}, we deduce that 
\[ |\bQ \bE^n - \bE^n|_{\LL^2}^2 = | \bQ \bu^n-\bu^n|_{\LL^2}^2 \leq C \, h^2 \, |\nabla \bu^n|_{\LL^2}^2.\]
Hence, \eqref{C(T,q)} with $q=1$  yields  for $\lambda >0$ 
 \begin{equation} 		\label{E_QE}
  \EE\Big(  1_{\tilde{\Omega}_{n-1}(M)} |\bE^n|_{\LL^2}^2 \Big) \leq 
(1+\lambda) \;  \EE\Big(  1_{\tilde{\Omega}_{n-1}(M)} |\bQ \bE^n|_{\LL^2}^2 \Big) + C(\lambda)  h^2.
\end{equation} 
Plugging this upper estimate in \eqref{pre-Gronwall-1}, 
 we deduce that for  $N$ large enough (that is $k$ small enough) to ensure
\begin{equation}		\label{constraint_Mk}
2\, (1+\lambda)\, k \Big[ 2+  C(\nu,\bar{C},  \lambda) \sqrt{M} +\tilde{\alpha}(\nu,\bar{C}, \lambda) M\Big]\leq \lambda,
\end{equation}
  we have 
\begin{align} 		\label{Gronwall-1} 
 \EE\Big( &1_{\tilde{\Omega}_{n-1}(M)} |\bE^n|_{\LL^2}^2 \!\Big) 
+ 
\lambda \;  \EE\Big( \!k \sum_{l=1}^n 
\! 1_{\tilde{\Omega}_{l-1}(M)} |\nabla \bE^l|_{\LL^2}^2\!\Big) 
 \leq  C(\nu,\lambda)   \big( k+ h^2\big) 
 \nonumber  \\
 &  +  C(\lambda) \,  h^2\; 
 \EE\Big(k  \sum_{l=1}^n \big| \nabla  \pi^l\big|_{L^2}^2\Big) 
 +  \EE\Big( k   \sum_{l=1}^{n-1}   1_{\tilde{\Omega}_{l-1}(M)} \;  \tilde{C}_2( M, \lambda) |\bE^l|_{\LL^2}^2\Big),
\end{align}
where using the Young inequality, we set
\begin{align*}
 \tilde{C}_2(M,\lambda )& =  
 \frac{(1+\lambda) \bar{C}^4}{ 2^4\,  (1-\lambda) \, (1-2\lambda)^3\, \nu^3} M + C(\nu, \bar{C}, \lambda) \sqrt{M}   + C(L_1\, \mbox{\rm Tr}Q, \lambda)\\
 &\leq   \frac{(1+2\lambda) \bar{C}^4}{ 2^4\,  (1-\lambda) \, (1-2\lambda)^3\, \nu^3} M +C(L_1\, \mbox{\rm Tr}Q, \nu, \bar{C}, \lambda)  . 
 \end{align*}
Neglecting the gradient term and using the discrete version of  Gronwall's lemma, we deduce for $\lambda \in (0, \frac{1}{2})$ 
\begin{align*}	
 \max_{0\leq n\leq N}& \EE\Big( 1_{\tilde{\Omega}_{n-1}(M)} |\bE^n|_{\LL^2}^2 \Big) \leq 
 \exp\big( T \tilde{C}_2(M,\lambda )\big) \;  \Big[ 
C(\nu)  \big( k+ h^2\big) +  
C\, h^2\,  \EE\Big(k  \sum_{l=1}^n \big| \nabla  \pi^l\big|_{L^2}^2\Big)\Big]. 
\end{align*} 
Plugging this upper estimate in \eqref{Gronwall-1}, we obtain  
\begin{align}		\label{loc-speed-gene-grad1}
 \EE\Big( \!k \sum_{l=1}^n  1_{\tilde{\Omega}_{l-1}(M)} |\nabla \bE^l|_{\LL^2}^2\Big) \leq& \; C(\nu)\;  \Big[ 1+ 
 \tilde{C}_2(M,\lambda) T e^{\tilde{C}_2(M,\lambda) T}\Big] \nonumber \\
&\times  \Big[  k+h^2 + { h^2}  \EE\Big(k  \sum_{l=1}^n \big| \nabla  \pi^l\big|_{L^2}^2\Big)\Big].
\end{align} 
\smallskip

We next study $\EE\big( \max_{1\leq n\leq m} 1_{\tilde{\Omega}_{n-1}(M)} |\bE^n|_{\LL^2}^2\big)$. Using \eqref{Ediff}, 
we deduce 
\begin{align}		
\frac{1}{2} \EE\Big( &\max_{1\leq n\leq m} 1_{\tilde{\Omega}_{n-1}(M)} |\bQ\bE^n|_{\LL^2}^2 \Big) 
+ \EE\Big( \sum_{l=1}^m 1_{\tilde{\Omega}_{l-1}(M) }  \Big[ 
\frac{1}{2} \big|\bQ(\bE^l - \bE^{l-1})\big|_{\LL^2}^2  + \nu  k  |\nabla \bE^l|_{\LL^2}^2\Big] \Big)\nonumber \\
 & \leq 
2\!  \sum_{i\in \{1,3\}} \!  \EE\big( R_i(m,k,h,M)\big)   
+ 2\EE\Big( \max_{1\leq n\leq m} R_2(n,k,h,M)\big),\quad \forall m=1, \cdots, N.
\end{align}

We rewrite $\tilde{R}_2(n,k,h,M)$ as a stochastic integral. For every $s\in [t_{l-1}, t_l)$ with $l=1, \cdots, N$, set $\underline{s}=
t_{l-1}$, and write accordingly $\bu(\underline{s})=\bu^{l-1}$,  $\bU(\underline{s})=\bU^{l-1}$, $\bE(\underline{s})=\bE^{l-1}$ and
$\tilde{\Omega}_{\underline{s}}(M) = \tilde{\Omega}_{l-1}(M)$. Then
\[ \tilde{R}_2(n,k,h,M)= \int_0^{t_n} 1_{ \tilde{\Omega}_{\underline{s}}(M)} \big( [G(\bu(\underline{s})) - G(\bU(\underline{s}))] dW(s) ,\, \bQ\bE(\underline{s}) \big)
\]
The Davies and Young inequalities, together with  the upper estimates $1_{\tilde{\Omega}_{l-1}(M)}  \leq 1_{\tilde{\Omega}_{l-1}(M)} $ for
$l=2, \cdots, N$,  and  \eqref{U0}, imply that for any  $\lambda >0$,  
\begin{align}		\label{upper_M1}
\EE\Big( & \max_{1\leq n\leq m} \tilde{R}_2(n,k,h,M) \Big)\\
 \leq& \,  3 \EE\Big( \Big\{ \int_0^{t_m} \!\! 1_{\tilde{\Omega}_{\underline{s}}(M)}
\|G(\bu(\underline{s}))-G(\bU(\underline{s}))\|_{{\mathcal L}(K,\WW^{1,2})}^2 {\rm Tr}\, Q \, |\bQ \bE(\underline{s})|_{\LL^2}^2 ds \Big\}^{\frac{1}{2}}
\Big) \nonumber \\
\leq &  3  \EE\Big(  \max_{2\leq n < m} 1_{\tilde{\Omega}_{l-1}(M)} |\bQ \bE^{l-1}|_{\LL^2} \Big\{ \sum_{l=2}^m L_1 1_{\tilde{\Omega}_{l-1}(M)}
|\bE^{l-1}|_{\LL^2}^2 k\, {\rm Tr}\, Q \Big\}^{\frac{1}{2}}\Big) \nonumber \\
&\qquad + 3 \sqrt{ L_1 {\rm Tr} \, Q}  \sqrt{k}\,  \EE(|\bE^0|_{\LL^2}^2)  \nonumber \\
\leq &\,  
  \lambda \,  \EE \Big(\!   \max_{1\leq n\leq m-1} \! 1_{\tilde{\Omega}_{l-1}(M)} |\bQ \bE^l|_{\LL^2}^2\Big) \!+ 
 \! \frac{9 L_1{\rm Tr}\,Q}{4\lambda }
\EE\Big( k \! \sum_{l=1}^{m-1} \!1_{\tilde{\Omega}_{l-1}(M)} |\bE^l|_{\LL^2}^2 \Big) \!+ \!C h^2 \sqrt{k}. 
\end{align}
Using the above upper estimates of $\EE( \max_{1\leq n\leq m} R_i(n,k,h,M))$, $i=1,2,3$,  an argument similar to the previous ones 
   yields for $\lambda \in (0, \frac{1}{2})$  and $N$ 
large enough 
\begin{align*}
\EE\Big( \max_{1\leq n \leq m} &1_{\tilde{\Omega}_{n-1}(M)} |\bE^n|_{\LL^2}^2\Big) \leq 
\EE\Big(k \sum_{l=1}^{m-1} 1_{\tilde{\Omega}_{l-1}(M)}  \tilde{C}_1(M,\lambda )  |\bE^l|_{\LL^2}^2\Big)
 \\
& +  C(\nu, \lambda)  \Big[ k+h^2 + h^2 \EE\Big( k\sum_{l=1}^m 1_{\tilde{\Omega}_{l-1}(M)} 
|\nabla \pi^l|_{\LL^2}^2 \Big)\Big], 
 \end{align*}
where 
\[ \tilde{C}_1(M,\lambda)  =   \frac{(1+2\lambda) \bar{C}^4}{2^4\, (1-2\lambda)^4 \nu^3} \, M + C(L_1\, \mbox{\rm Tr}Q,\lambda, \nu, \lambda).
\]
Using once more the discrete Gronwall lemma, we obtain
\begin{align*}
\EE\Big( \max_{1\leq n\leq m} 1_{\tilde{\Omega}_{n-1}(M)} |\bE^n|_{\LL^2}^2\Big) \leq &\;  C(\nu, \lambda) 
\Big[ k+h^2 + h^2 \, \EE\Big( k\sum_{l=1}^m  
|\nabla \pi^l|_{\LL^2}^2 \Big)\Big]\; e^{T \tilde{C}_1(M,\lambda)}.  
\end{align*}

 Let $\delta\in (0,1)$; one may choose $\lambda \in (0,\frac{1}{2})$ small enough to have $\frac{1+2\lambda}{(1-2\lambda)^4} \leq  (1+\delta)$. 
Since the sets $\{ \tilde{\Omega}_{l}(M)\}_l$ decrease,  \eqref{loc_mom_gene} holds with $\tilde{C}_1(M)$ defined by \eqref{C1(M)gene}.

The upper estimate  \eqref{loc_mom_gene_nabla} is a straightforward consequence of  \eqref{loc-speed-gene-grad1} 
for $\lambda \in (0,\frac{1}{2})$ small enough.

Finally, observe that for ``large $M$", we may replace $\tilde{C}_1(M)$ by $ (1+\delta) \frac{\bar{C}^4}{2^4 \nu^3} M $, provided that the constraint
\eqref{constraint_Mk}, which enables us to use the discrete Gronwall lemma, is satisfied for some constant $\lambda$ defined in terms of $\delta$. 
This completes the proof. 
\end{proof} 

A similar argument proves the following localized convergence result for Algorithm 2.
\begin{prop}		\label{loc-cv-divfree}
 Let $(\bU^l)_l$ be solution of Algorithm 2,  that is ${\bf V}_h\subset V$,  and 
 $G$ satisfy the growth and Lipschitz conditions {\bf (G2)}. Let $u_0\in L^8(\Omega; V)$ and $\bU^0\in L^8(\Omega; \HH_h)$ be 
${\mathcal F}_0$-measurable random variables such that \eqref{U0} holds. For every $M>0$ and $l=0, \cdots, N$, set
\begin{equation}		\label{OmegaM}
{\Omega}_l(M):= \Big\{ \omega \in \Omega \, : \, \max_{0\leq j\leq l} \|\bu^j\|_V^2 \leq M\Big\}.
\end{equation}
 There exists  $C_0>0$ such that for   $k= \frac{T}{N}>0$ small enough to have $k M\leq C_0$, we have for every $\delta>0$ 
and for every $m=1, \cdots, N$ 
\begin{equation}		\label{loc_mom_divfree}
\EE\Big( 1_{\Omega_{m-1}(M)} \max_{1\leq n\leq m}  |\bE^n|_{\LL^2}^2 \Big)  \leq \; 
C(\nu,\delta)  \Big( k+ h^2\Big)\,  e^{\tilde{C}_3(M)T} , 
\end{equation} 
where 
\begin{align}		\label{C3(M)divfree}
  \tilde{C}_3(M):= &\,   (1+\delta)\; \Big[ \frac{ \bar{C}^2}{4 \nu} +1\Big] M  +  C(L_1\, \mbox{\rm Tr}\, Q, \delta) \nonumber\\
   \equiv & \,  (1+\delta)\,  \Big[ \frac{ \bar{C}^2}{4 \nu} +1\Big] M \quad \mbox{\rm for ``large"}\; M,  
\end{align}
and $\bar{C}$ is the constant appearing in the Gagliardo-Nirenberg inequality \eqref{interpol}. Furthermore, 
\begin{equation}		\label{loc_mom_divfree_nab}
\EE\Big( k \sum_{l=1}^m 1_{\Omega_{l-1}(M)} |\nabla \bE^l|_{\LL^2}^2 \Big) \leq 
C(\nu, \delta)  \big( k+ h^2\big) \,  T \tilde{C}_3(M) \, e^{\tilde{C}_3(M) T} .
 \end{equation} 
 \end{prop}
\begin{proof}
We briefly sketch the argument. The identities \eqref{E^m-E^(m-1)} with $b$ instead of $\tilde{b}$ and   ${\rm div}\,  \bQ\bE^l=0$, the estimates 
  \eqref{maj_E-E} and  \eqref{E_QE} imply for $n,m=1, \cdots, N$
\begin{align*}
&\EE\Big( 1_{\Omega_{n-1}(M)} | \bQ  \bE^n|_{\LL^2}^2 \Big) 
+  2  \nu \EE\Big( \sum_{l=1}^n k 1_{\Omega_{l-1}(M)} |\nabla \bE^l|_{\LL^2}^2\Big) \leq  2 
\EE\big[ {R}_1(n,k,h,M) + {R}_3(n,k,h,M) \big],\\
& \EE\Big( \max_{1\leq n\leq m}1_{\Omega_{n-1}(M)} | \bQ  \bE^n|_{\LL^2}^2\Big) \leq 2 \EE \Big[ {R}_1(n,k,h,M) + {R}_3(n,k,h,M) +
\max_{1\leq n\leq m} {R}_2(n,k,h,M)\Big], 
\end{align*}
where 
  ${R}_j(n,k,h,M)$ is deduced from $\tilde{R}_j(n,k,h,M)$ replacing $\tilde{\Omega}_{l-1}(M)$ by $\Omega_{l-1}(M)$ for $j=1,2,3$ (and where in 
  ${R}_2(n,k,h,M)$ the term containing the pressure is omitted). 

The upper estimates \eqref{U0},  \eqref{E1}, \eqref{b(u-u)},   \eqref{b-b}  and \eqref{upper_M3} with $\Omega_{l-1}(M)$ instead of 
$\tilde{\Omega}_{l-1}(M)$ imply for $\epsilon_1\in (0, \frac{1}{2})$ and $\epsilon>0$
\begin{align*}
 \EE\Big( &1_{\Omega_{n-1}(M)} |\bQ \bE^n|_{\LL^2}^2\Big) +  2 \nu  \EE\Big( k\sum_{l=1}^n 1_{\Omega_{l-1}(M)} |\nabla \bE^l|_{\LL^2}^2\Big)
\leq \; C(\nu, \epsilon, \epsilon_1)(k+h^2) \\
&+ C(\nu, \epsilon_1, M)  \EE\Big( k\, 1_{\Omega_{n-1}(M)} |\bE^n|_{\LL^2}^2\Big) 
+  2(10\epsilon + 2 \epsilon_1)   \nu  
\EE\Big( k\sum_{l=1}^n 1_{\Omega_{l-1}(M)} |\nabla \bE^l|_{\LL^2}^2\Big)\\
& +  
\EE\Big( k \sum_{l=1}^{n-1} 1_{\Omega_{l-1}(M)}| \Big[ C(L_1 \, \mbox{\rm Tr}\, Q) +  \Big( 1+\frac{\tilde{C}^2}{2^3 \epsilon_1 \nu} \Big)  M  \Big]
 \bE^l|_{\LL^2}^2\Big).
\end{align*}
Let $\lambda \in \big( 0, 1)$; set $\epsilon_1=\frac{1}{2}(1-\lambda)$ and $\epsilon = \frac{\lambda}{20}$; then $2\epsilon_1+10\epsilon = 1-\frac{\lambda}{2}$. 
Using \eqref{E-QE}, 
 we deduce for $\lambda \in (0,1)$  and $N$ large enough (that is $k$ small enough)
 \begin{align} 		\label{Gronwall-1-Bis} 
 \EE\Big( 1_{{\Omega}_{n-1}(M)} |\bE^n|_{\LL^2}^2 \!\Big) &
+ 
\lambda \;  \EE\Big( \!k \sum_{l=1}^n 
\! 1_{{\Omega}_{l-1}(M)} |\nabla \bE^l|_{\LL^2}^2\!\Big) 
 \leq  C(\nu,\lambda)   \big( k+ h^2\big)  \nonumber  \\
 &  +  \EE\Big( k \sum_{l=1}^{n-1}  1_{\tilde{\Omega}_{l-1}(M)} \;  \tilde{C}_4(M, \nu, \lambda) |\bE^l|_{\LL^2}^2\Big),
\end{align}
 where 
\[ \tilde{C}_4(M,\nu,\lambda)):=   (1+\lambda)  \Big( \frac{ \bar{C}^2}{2^2 (1-\lambda)^2  \, \nu}+1\Big)   M
 + C(L_1\, \mbox{\rm Tr}\, Q, \lambda).\]
Using the discrete version of Gronwall's lemma,  we obtain an upper estimate of $\EE\big( 1_{\Omega_{n-1}(M)} |\bE^l|_{\LL^2}^2\big)$.
 Plugging this 
result in the above upper estimate \eqref{Gronwall-1-Bis}, we deduce 
\begin{equation}		\label{upper_nabla_divfree}
\EE\Big( \!k \sum_{l=1}^N \!1_{\Omega_{l-1}(M)} |\nabla \bE^l|_{\LL^2}^2\Big) \leq C(\nu, \lambda) (k+h^2) 
 T \tilde{C}_4(M,\nu,\lambda)\;  \exp\Big( T \tilde{C}_4(M,\nu,\lambda)\Big) . 
\end{equation}
On easily sees that the upper estimate  \eqref{upper_M1} holds for  $R_2(n,k,h,M)$ and $\Omega_{l-1}(M)$ instead of 
$\tilde{R}_2(n,k,h,M)$ and $\tilde{\Omega}_{l-1}(M)$ respectively. Hence, we deduce that for $\lambda \in \big(0,\frac{1}{2}\big)$ and $k$ small enough, we have
 \begin{align*}
\EE\Big( \max_{1\leq n\leq m} 1_{\Omega_{n-1}(M)} |\bE^n|_{\LL^2}^2 \Big) \leq &\; 
C(\nu, \lambda) (k+h^2) 
+ \EE\Big( \tilde{C}_3(M, \lambda)
\; k\sum_{l=1}^{m-1} 1_{\Omega_{l-1}(M)}   |\bE^l|_{\LL^2}^2 \Big),
\end{align*}
where 
\[ \tilde{C}_3(M,\lambda):=  (1+\lambda) \Big[ \frac{ \bar{C}^2}{2^2\, (1-2\lambda) (1-\lambda) \nu} +1\Big]M  + C( L_1 \mbox{\rm Tr}\, Q, \lambda). \]
The discrete Gronwall lemma implies for $m=1, \cdots, N$
\[ \EE\Big( \max_{1\leq n\leq m} 1_{\Omega_{n-1}(M)} |\bE^n|_{\LL^2}^2 \Big) \leq C(\nu, \lambda) (k+h^2) 
\;   \exp\big( T \tilde{C}_3(M, \lambda)\big).
\]
Fix $\delta>0$; one can choose $\lambda \in \big( 0, \frac{1}{2}\big)$ small enough to ensure $\frac{1+\lambda}{(1-\lambda)(1-2\lambda)} \leq 1+\delta$.  
 Since the sets $\{\Omega_l(M)\}_l$ are decreasing,, this
 concludes the proof of \eqref{loc_mom_divfree}. The proof of \eqref{loc_mom_divfree_nab}, which is similar to that   in the proof of
  \eqref{loc_mom_gene_nabla}, is omitted.  
 \end{proof}

\section{A general framework for strong convergence}		\label{framework}
This section is devoted to a very general setting to prove $L^2(\Omega)$-convergence for a real-valued family $\{ X(\eta)\}_\eta$ of  random variables 
with a specific rate of convergence. These results will be applied in the next section to obtain rates of convergence of $|\bE_n|_{\LL^2}$. 

We first prove that the speed of convergence of $\{X(\eta)\}_\eta$ can be deduced from localized $L^2(\Omega)$ estimates of $X(\eta)$ in terms if another
family $\{Y(\eta)\}_\eta$, and from a control of moments of $X(\eta)$ and $Y(\eta)$ uniformly in $\eta$.

\begin{theorem}		\label{general_strong}
Let $\{X(\eta)\}_\eta$ and $\{Y(\eta)\}_\eta$ be families of non-negative random variables indexed by some parameter $\eta\in (0,\eta_0]$, and let
$\varphi : (0,\eta_0] \to \RR$ be  a function such that $\varphi(\eta)\to 0$ as $\eta \to 0$. Suppose that for some exponents $a\geq 1$ and $p\in (1,\infty)$, we have
\begin{itemize}
\item Let  $M_0$, $C_0$ be positive constants, $\psi:(0,\eta_0] \to [0,\infty)$ be a function such that 
$\psi(\eta) \to 0$  as $\eta \to 0$. There exists a positive constant $C_1$ such that  for $M\geq M_0$ and $\eta \in (0,\eta_0]$ such that
  $\psi(\eta) M\leq C_0$, we have for any $\delta\in (0,1)$
\begin{equation}		\label{eq6.1}
\EE\Big( X(\eta) 1_{\{ Y(\eta)^a\leq M\} }\Big) \leq \varphi(\eta) \exp\big[ (1+\delta)C_1 M\big]; 
\end{equation}
\item
\begin{equation}		\label{eq6.2}
\sup_{\eta\in (0, \eta_0]} \EE\big( X(\eta)^p \big) = C(p)<\infty.
\end{equation}
\end{itemize}

(i) Suppose that 
\begin{equation}		\label{eq6.3}
\sup_{\eta\in (0,\eta_0]} \EE\big( Y(\eta)^q\big) = \tilde{C}(q)<\infty
\end{equation}
for some exponent $q>0$. Then for $\eta$ small enough
\begin{equation}		\label{eq6.4}
\EE\big( X(\eta) \big) \leq C \, \big| \ln\big( \varphi(\eta)\big) \big|^{-\frac{q(p-1)}{ap}}.
\end{equation}
If condition \eqref{eq6.2} holds for any $p\in (1,\infty)$, the upper bound of \eqref{eq6.4} can be replaced by $C   \big| \ln\big( \varphi(\eta)\big) \big|^{-\gamma_1}$ for
any $\gamma_1 < \frac{q}{a}$.  
\smallskip

(ii) Suppose that $a>1$ and that for some positive constant $\tilde{\alpha}_0$ 
\begin{equation}		\label{eq6.5}
\EE\big[ \exp\big(\alpha Y(\eta)\big) \big] =\bar{C}(\alpha)<\infty, \quad \mbox{\rm for}\; \alpha \in (0,\tilde{\alpha}_0)
\end{equation}
Then for $\eta$ small enough
\begin{equation}		\label{eq6.6}
\EE\big( X(\eta)\big) \leq C \, \exp\big( -\gamma_2 \big| \ln\big( \varphi(\eta)\big) \big|^{\frac{1}{ a}}\big)
\quad \mbox{\rm for}\; \gamma_2 < \frac{p-1}{p} \, \tilde{\alpha}_0\,  C_1^{-\frac{1}{a}}.
\end{equation}
Furthermore, if \eqref{eq6.2} holds for every $p\in (1,\infty)$, then the upper estimate \eqref{eq6.6} holds for $\gamma_2 < {\tilde{\alpha}_0} C_1^{-\frac{1}{a}}$.
\smallskip

(iii) Suppose that $a=1$ and that \eqref{eq6.5} holds. Then for $\eta$ small enough we have
\begin{equation}		\label{eq6.7}
\EE\big( X(\eta)\big) \leq \varphi(\eta)^\gamma \quad \mbox{\rm for} \; \gamma_3 <  \frac{\tilde{\alpha}_0 (p-1)}{\tilde{\alpha}_0 (p-1)+C_1 p}.
\end{equation}
Furthermore, if \eqref{eq6.2} holds for every $p\in (1,\infty)$, then the upper estimate \eqref{eq6.7} holds for 
$\gamma_3 < \frac{\tilde{\alpha}_0}{\tilde{\alpha}_0+C_1}$.
\end{theorem}
\begin{proof} For $\eta >0$ let $M(\eta)>0$, $\Omega(\eta):= \{ Y(\eta)^a \leq M(\eta)\}$.
 
(i) Using \eqref{eq6.3} we deduce
\[ 		
 \PP\big( \Omega(\eta)^c\big)  = \PP\big( Y(\eta)^q > M(\eta)^{\frac{q}{a}} \big) \leq C(q)\, M(\eta)^{-\frac{q}{a}}.
 \] 
Furthermore, \eqref{eq6.2} and H\"older's inequality with conjugate exponents $p$ and $\frac{p}{p-1}$ imply
\begin{equation}	\label{eq6.8}	
 \EE\big( 1_{\Omega(\eta)^c}\, X(\eta) \big) \leq \big\{ \EE\big( X(\eta)^p\big)\big\}^{\frac{1}{p}} \big\{ \PP\big(\Omega(\eta)^c\big) \big\}^{\frac{p-1}{p}} 
\leq C(p,q) M(\eta)^{-\frac{q(p-1)}{ap}}.
\end{equation}
Let $M(\eta) \geq M_0$ and $\eta$ be small enough to have $\varphi(\eta)<1$ and $\eta M(\eta)  \leq C_0$; using \eqref{eq6.1} we deduce 
\begin{equation} \label{eq6.9}
 \EE\big( 1_{\Omega(\eta)} X(\eta) \big) \leq \varphi(\eta) \exp\big[ (1+\delta)\, C_1\, M(\eta)\big]\quad \mbox{\rm for some}\; \delta\in (0,1).
 \end{equation}
Choose $M(\eta)$ such that, up to some multiplicative constant, the right hand sides of \eqref{eq6.8} and \eqref{eq6.9} agree. Taking logarithms, this
comes down to 
\[(1+\delta) \,  C_1\, M(\eta) - \big|\ln \big( \varphi(\eta)\big)\big|  =
-c\, \ln\big( M(\eta)\big) + \bar{c}(\eta)\]
where $c= \frac{q(p-1)}{ap}$,  and $\bar{c}_1\leq \bar{c}(\eta) \leq \bar{c}_2$ holds for any $\eta\in (0,\eta_0]$, and some constants $\bar{c}_1, \bar{c}_2$.  
Set 
\[ M(\eta):= \frac{1}{(1+\delta)\, C_1 } \big[ \big|\ln \big( \varphi(\eta)\big)\big|  - c \ln\big( \big|\ln \big( \varphi(\eta)\big)\big| \big)\big]
\sim C \big|\ln \big( \varphi(\eta)\big)\big|  \]
for some positive constant $C$. Then $\varphi(\eta) M(\eta)\to 0$ as $\eta \to 0$; furthermore, for some constant $\bar{C}$ we have
\[ C_1(1+\delta) M(\eta) - \big|\ln \big( \varphi(\eta)\big)\big|  +c\ln\big(M(\eta) \big) = c\ln\Big( \frac{M(\eta)}{\big|\ln \big( \varphi(\eta)\big)\big| }\Big) 
\to \bar{C} \quad \mbox{\rm as} \; \eta \to 0.\] 
Hence, for $\eta$ small enough,  both upper estimates \eqref{eq6.8} and \eqref{eq6.9} agree - up to some multiplicative constant - with 
$C\; \big|\ln \big( \varphi(\eta)\big)\big|^{-\frac{q(p-1)}{ap}}$; this concludes the proof of \eqref{eq6.4}. If \eqref{eq6.2} holds for any $p\in (1,\infty)$, the bound on $\gamma_1$  is
a straightforward consequence of \eqref{eq6.4}. 
\smallskip

(ii) Using \eqref{eq6.5} and the Markov inequality, we deduce for $\alpha \in (0,\tilde{\alpha}_0)$ 
\[ \PP\big(\Omega(\eta)^c\big) = \PP\big[ \exp\big( \alpha Y(\eta)\big) > \exp\big( \alpha M(\eta)^{\frac{1}{a}}\big) \big] 
\leq C(\alpha) \exp\big( -\alpha M(\eta)^{\frac{1}{a}} \big).\]
Using again \eqref{eq6.2} and H\"older's inequality,  we deduce 
\begin{equation} 		\label{eq6.10}
\EE\big( 1_{\Omega(\eta)^c} \, X(\eta)\big) \leq C(p,\alpha) \exp\Big( - \alpha \frac{p-1}{p} M(\eta)^{\frac{1}{a}}\Big).
\end{equation}
We next choose $M(\eta)=\frac{1}{C} \big|\ln \big( \varphi(\eta)\big)\big| $ for some constant $C>(1+\delta) C_1$. 
Then $\varphi(\eta) M(\eta)\to 0$ as $\eta\to 0$; thus we may
use \eqref{eq6.1} for small $\eta$.  For this choice of $M(\eta)$, we have 
\[ \EE\big( 1_{\Omega(\eta)} \, X(\eta) \big) \leq \exp\Big[ - \Big( 1-\frac{(1+\delta) C_1}{C}\Big) \big|\ln \big( \varphi(\eta)\big)\big|  \Big].
\] 
Since $a>1$, we have $M(\eta)\to \infty$ as $\eta\to 0$; thus,  the right hand side of \eqref{eq6.10} converges to 0 slower than the above one.
 Splitting $\EE\big( X(\eta)\big)$ on $\Omega(\eta)$ and
$\Omega(\eta)^c$, we deduce that the largest term in this sum is the expected value on $\Omega(\eta)^c$. Hence, using \eqref{eq6.10} for this choice
of $M(\eta)$, we deduce that for $\eta$ small enough, 
\[ \EE\big( X(\eta)\big)\leq C \exp\Big( -\alpha \frac{p-1}{p} C^{-\frac{1}{a}} \big|\ln \big( \varphi(\eta)\big)\big|^{\frac{1}{a}}\Big) \quad 
\mbox{\rm for}\; \delta\in (0,1)\;  \mbox{\rm and}\; C>(1+\delta) C_1.\]
Fix $\gamma_2<\tilde{\alpha}_0\; \frac{p-1}{p} C_1^{-\frac{1}{a}}$; then choose $\alpha<\tilde{\alpha}_0$ close to $\tilde{\alpha}_0$, $\delta>0$ very small and
$C>C_1(1+\delta)$ very close to this bound to have
\[ \gamma_2 < \alpha \frac{p-1}{p} \big[ (1+\delta)\, C_1 \big]^{-\, \frac{1}{a}} < \tilde{\alpha}_0 \frac{p-1}{p} C^{-\, \frac{1}{a}}.\]
This inequality clearly yields \eqref{eq6.6}. 

Note that if \eqref{eq6.2} holds for all $p\in (0,1)$, given  $\gamma_2< \tilde{\alpha}_0 C_1^{-\frac{1}{a}}$ we can choose $p$ large enough to have
$\gamma_2< \alpha \frac{p-1}{p} C_1^{-\frac{1}{a}}$.
\smallskip

(iii) Using the upper estimates \eqref{eq6.1} with $a=1$ and \eqref{eq6.10}, we choose $M(\eta)$ such that the right hand sides of both inequalities agree
up to some multiplicative constant. Taking logarithms, we deduce that $M(\eta)$ should be such that
\[ (1+\delta)\, C_1 M(\eta) - \big|\ln \big( \varphi(\eta)\big)\big|  =
 -\alpha \frac{p-1}{p} M(\eta) \quad \mbox{\rm for some }\; \delta \in (0,1) \; \mbox{\rm and}\; \alpha \in (0,\tilde{\alpha}_0).\]
Thus, if we set $\beta:= \alpha \frac{p-1}{p}$ and  $M(\eta):= \frac{\big|\ln \big( \varphi(\eta)\big)\big| }{C_1(1+\delta)+\beta}$, 
we have $\varphi(\eta) M(\eta)\to 0$ as $\eta \to 0$, and
for $\eta$ small enough, $\varphi(\eta)<1$; thus for $\eta$ small enough, 
\[ \EE\big( X(\eta)\big) = C \exp\Big( - \frac{\beta \big|\ln \big( \varphi(\eta)\big)\big| }{\beta + (1+\delta) C_1}\Big) =C  \varphi(\eta)^{\frac{\beta}{C_1(1+\delta)+\beta}}
\quad \mbox{\rm for }\; \alpha \in (0,\alpha_0)\; \mbox{\rm and} \; \delta\in (0,1).
\] 
Set $\beta_0:= \tilde{\alpha}_0 \frac{p-1}{p}$ and let $\gamma_3< \frac{\beta_0}{\beta_0+C_1}$; choose $\alpha<\tilde{\alpha}_0$ close enough to
 $\tilde{\alpha}_0$ and $\delta>0$
small enough to have $\gamma_3<\frac{\beta}{\beta + (1+\delta)C_1}<\frac{\beta_0}{\beta_0+C_1}$. We deduce  \eqref{eq6.7}. 

If \eqref{eq6.2} is satisfied for any $p\in (0,1)$, given $\gamma_3<\frac{\tilde{\alpha}_0}{\tilde{\alpha}_0+C_1}$ we can choose $p$ large enough to have 
$\gamma_3<\frac{\beta_0}{\beta_0+C_1}<\frac{\tilde{\alpha}_0}{\tilde{\alpha}_0+C_1}$; 
therefore, we conclude that \eqref{eq6.7} holds with the exponent $\gamma_3$ in the right hand side.
\end{proof}

We next deduce similar results if the assumption \eqref{eq6.1} is slightly weakened as follows.
\begin{coro} 		\label{general_cor}
Let $\{Z(\eta)\}_\eta$ and $\{Y(\eta)\}_\eta$ be families of non-negative random variables indexed by some parameter $\eta\in (0,\eta_0]$, and let
$\varphi : (0,\eta_0] \to \RR$ be  a function such that $\varphi(\eta)\to 0$ as $\eta \to 0$. Suppose that for some exponents $a\geq 1$ and $p\in (1,\infty)$
\begin{itemize}
\item Let  $M_0$, $C_0$ be positive constants, $\psi:(0,\eta_0] \to [0,\infty)$ be a function such that 
$\psi(\eta) \to 0$  as $\eta \to 0$. There exists a positive constant $C_1$ such that  for $M\geq M_0$ and $\eta \in (0,\eta_0]$ such that
  $\psi(\eta) M\leq C_0$, we have for any $\delta\in (0,1)$
\begin{equation}		\label{eq6.1Bis}
\EE\Big( Z(\eta) 1_{\{ Y(\eta)^a\leq M\} }\Big) \leq \varphi(\eta) \,  C_1\,  M\, \exp\big[ (1+\delta)C_1 M\big],
\end{equation}
\item
\begin{equation}		\label{eq6.2Bis}
 \sup_{\eta\in (0, \eta_0]} \EE\big( Z(\eta)^p \big) = C(p)<\infty. 
\end{equation}
\end{itemize}

(i) Let $\{Y(\eta)\}_\eta$ satisfy the assumption \eqref{eq6.3}. Then for $\eta$ small enough,  \eqref{eq6.4} holds for $Z(\eta)$ instead of $X(\eta)$. 

(ii) Let $\{Y(\eta)\}_\eta$ satisfy the assumption \eqref{eq6.5} and $a>1$. Then for $\eta$ small enough, \eqref{eq6.5} holds with $Z(\eta)$ instead of $X(\eta)$.

(iii) $\{Y(\eta)\}_\eta$ satisfy the assumption \eqref{eq6.5} and $a=1$. Then for $\eta$ small enough, \eqref{eq6.6} holds with $Z(\eta)$ instead of $X(\eta)$.
\end{coro} 
\begin{proof}
The proof is a straightforward consequence of the arguments used in the proof of   Theorem \ref{general_strong}, and, given any $\delta>0$, of the inequality 
$x \exp\big[ (1+\delta)x\big] \leq \exp\big[ (1+2\delta)x\big]$ valid for large enough $x>0$. 
\end{proof}

\section{ Convergence in $L^2(\Omega)$} 		
\label{s5}
In this section, we prove a speed of convergence in $L^2(\Omega)$-norm of the difference $\bE^l$ between the time and 
space-time discretizations. 
Coupled with the results of Theorems 
\ref{th_Euler_lingrow} and \ref{th_Euler_exp}, this will provide a strong (that is $L^2(\Omega)$) speed of convergence of the space-time Euler
scheme $\bU^l$ to $u(t_l)$ in $\LL^2(D)$ uniformly on the time grid $t_l=l \frac{T}{N}$.

\subsection{ Strong convergence of Algorithm 1} 
In this subsection, we focus   on the solution defined in Algorithm 1, that is  \eqref{algo_FE} and \eqref{divU}.
We will give results in the  case of a  multiplicative and of an  additive stochastic perturbation.
\subsubsection{Multiplicative noise}
We suppose that the diffusion coefficient $G$ satisfies the growth and Lipschitz conditions {\bf (G1)} or {\bf (G2)}.
 The following theorem is the main result of this section for general finite elements and diffusion coefficients.
As in \cite{CarPro}, for general coefficients we have either to keep the gradient of the pressure  (with some growth conditon), 
or to impose some relation between the space mesh $h$ and the time mesh $k=\frac{T}{N}$. 
\begin{theorem}		\label{th-str-gen-mul}
 Fix an integer $q_0\geq 3$;  suppose that $u_0\in L^{2^{q_0}}(\Omega;V)$ and $U_0\in L^{2^{q_0}}(\Omega;L^2_{per})$.

(i)  Let $G$ satisfy  condition {\bf (G1)} and suppose that $h^2 \bE\Big( k\sum_{l=1}^N |\nabla \pi^l|_{L^2}^2\Big)\to 0$ as $k,h\to 0$; 
then for $k+h^2+h^2 \bE\Big( k\sum_{l=1}^N |\nabla \pi^l|_{L^2}^2\Big) $  small enough, 
\begin{align}			\label{speed-gene-G1}
\EE\Big( \max_{ 0\leq l\leq N } |u(t_l) - \bU^l|_{\LL^2}^2 &
+ k  \sum_{ l=1}^N |\nabla u(t_l) - \nabla \bU^l|_{\LL^2}^2 \Big) \nonumber \\
 &\leq  C 
\Big|  \ln \Big[ k+h^2 + h^2 \bE\Big( k\sum_{l=1}^N |\nabla \pi^l|_{L^2}^2\Big) \Big]\, \Big|^{-(2^{q_0-2}-\frac{1}{2})}.
\end{align}

(ii)  Let $G$ satisfy  condition {\bf (G1)} and suppose that as $h,k\to 0$ we have $h^2 k^{-1}\to 0$; 
then for $k$ and $h^2\,  k^{-1}$ small enough, 
\begin{align}			\label{speed-gene-G1Bis}
\EE\Big( \max_{ 0\leq l\leq N } |u(t_l) - \bU^l|_{\LL^2}^2 
+ k  \sum_{ l=1}^N |\nabla u(t_l) - \nabla \bU^l|_{\LL^2}^2 \Big) &\leq  C 
\big|  \ln \big[ k+h^2 k^{-1} \big]\; \big|^{-(2^{q_0-2}-\frac{1}{2})}.
\end{align}

(iii) Suppose that  $G$ satisfies condition {\bf (G2)}. Then for $k$ and $h$ small enough  (without any restriction), we have
\begin{align}			\label{speed-gene-G2}
\EE\Big( \max_{ 0\leq l\leq N } |u(t_l) - \bU^l|_{\LL^2}^2 
+ k  \sum_{ l=1}^N |\nabla u(t_l) - \nabla \bU^l|_{\LL^2}^2 \Big) &\leq  C 
\big|  \ln \big(k+h^2 \big)  \big|^{-(2^{q_0-2}-\frac{1}{2})}.
\end{align}
 \end{theorem}
\begin{remark} If $u_0$ and $\bU^0$ are deterministic  or have moments of all orders (such as Gaussian random variables), in all cases the exponent
of the logarithm is arbitrary large.
\end{remark}
\begin{proof}
(i) Let $\eta=k+h^2$, $\varphi(\eta)=k+h^2+ h^2\EE\big( k\sum_{l=1}^N |\nabla \pi^l|_{L^2}^2\big)$ and $\psi(\eta)=k\leq \eta$; 
 by assumption, $\varphi(\eta)\to 0$ as $\eta\to 0$.
Set
\begin{equation}		\label{XYZ}
 X(\eta)= \max_{0\leq l\leq N} |\bu^l-\bU^l|_{\LL^2}^2, \;  Y(\eta)=\max_{0\leq l\leq N-1} \| \bu^l\|_V^2 \;\mbox{\rm and} \;  Z(\eta)=k\sum_{l=1}^N
| \nabla \bu^l-\nabla \bU^l |_{\LL^2}^2. 
\end{equation}
Using \eqref{C1(M)gene} in Proposition \ref{loc-conv}, we deduce that \eqref{eq6.1} and \eqref{eq6.1Bis} hold
 for $a=2$ and $C_1=\frac{\bar{C}^4\, T}{2^4\nu^3}$. Furthermore,
the upper estimates  \eqref{C(T,q)}, 
\eqref{U2} and \eqref{sum_grad}  imply that \eqref{eq6.2} and \eqref{eq6.2Bis} are true for 
$p=2^{q_0-1}$, while \eqref{C(T,q)} implies that \eqref{eq6.3} holds for
$q=2^{q_0-1}$.  Since $\frac{q(p-1)}{aq}=2^{q_0-2}-\frac{1}{2}$, using part (i) in Theorem \ref{general_strong} and Corollary \ref{general_cor}, we deduce 
\[ 
  \EE\Big( \max_{ 0\leq l\leq N } |u^l - \bU^l|_{\LL^2}^2 \Big)  + k  \sum_{ l=1}^N |\nabla u^l - \nabla \bU^l|_{\LL^2}^2 \Big)
\leq  C 
\Big|  \ln \Big[ k+h^2 + h^2 \bE\Big( k\sum_{l=1}^N |\nabla \pi^l|_{L^2}^2\Big) \Big]\, \Big|^{-(2^{q-2}-\frac{1}{2})}.
\] 
Coupling this upper estimate with  \eqref{speed_lin_growth_Euler}, and using the inequality  inequality $k^{-(2^{-q-1}-1)} \leq C  k^{-(2^{q-2}-1/2)}$
for ``small" $k$,  we obtain  \eqref{speed-gene-G1}.

 \smallskip
 
 (ii)  Using the upper estimate 
   \eqref{pressure-W}   
we deduce 
\[ k+h^2+ h^2 \bE\Big( k\sum_{l=1}^N |\nabla \pi^l|_{L^2}^2\Big) \leq k+h^2+C(T) \frac{h^2}{k} \leq C(T)\Big( k+  \frac{h^2}{k}\Big),\]
where $k+h^2/k\equiv 0$. Hence, \eqref{speed-gene-G1} implies \eqref{speed-gene-G1Bis}.

 (iii) Recall that if $G$ satisfies condition {\bf (G2)}, \eqref{pressure-V} implies 
  \[ k+h^2+ h^2 \bE\Big( k\sum_{l=1}^N |\nabla \pi^l|_{L^2}^2\Big) \leq k+h^2\big(1+ C(T)\big)  \leq C(T) (k+h^2).\]
 Thus \eqref{speed-gene-G1} implies \eqref{speed-gene-G2}. 
  \end{proof}
 
\subsubsection{ Additive noise }
As in \cite{Be-Mi_time}, we replace the H\"older  inequality for a power  of $\max_l \|\bu^l\|^2_V$ by an exponential
Markov inequality for $\max_l \|\bu^l\|_V^2$;  recall that the maximal exponent of the exponential moments proved in 
Theorem \ref{exp-mom} is $\tilde{\alpha}_0$.
This yields the following 
 \begin{theorem}		\label{th-str-gene-add}
Let $G$ satisfy condition {\bf (G1)} with $K_1=0$, i.e.,  $\| G(u)\|_{{\mathcal L}(K,\WW^{1,2})}^2 \leq K_0$. Let  $u_0\in V$
and $U_0\in L^{2^{q_0}}(\Omega;L^2_{per})$ for some $q_0\in [3,\infty)$. 
Set   
 $\tilde{\alpha}_0:= \frac{\nu}{2\, \tilde{C} \, K_0 \,  {\rm Tr} Q}$,  where $\tilde{C}$ is defined by 
$|\nabla u|_{\LL^2}^2 \leq \tilde{C} |Au|_{\LL^2}^2$.  Let 
$\kappa_0:= \frac{2^{q_0-1}-1}{2^{q_0-1}} \, \tilde{\alpha_0}\,  \frac{4}{\bar{C}^2} \, \sqrt{\frac{\nu^3}{T}}$.

(i) 
 Suppose that
$k+h^2+ h^2 \EE\big(k \sum_{l=1}^N |\nabla \pi^l|_{L^2}^2\big) $ is ``small". Then  for $\gamma <  \kappa_0$ 
\begin{align}			\label{speed-gen-add}
\EE\Big(\! \max_{  0\leq l\leq N } |u(t_l) - \bU^l|_{\LL^2}^2 & + \! k  \sum_{ l=1}^N |\nabla u(t_l) - \nabla \bU^l|_{\LL^2}^2 \!\Big)\nonumber \\
 &\leq  C \exp\Big( \!
- \gamma \, 
  \Big| \ln \Big[ k+h^2+h^2\EE\Big( k\sum_{j=1}^N |\nabla \pi^j|_{L^2}^2 \Big) \Big] \Big|^{\frac{1}{2}}\Big). 
\end{align}

(ii) Suppose that  $G$ satisfies  condition {\bf (G1)} and that $h^2\,  k^{-1}\to 0$ as $h,k\to 0$; then for ``small"  $k+ h^2 k^{-1}$ we have for $\gamma < \kappa_0$
 \begin{align}			\label{speed-gen-add_Bis}
\EE\Big(\! \max_{  0\leq l\leq N } |u(t_l) - \bU^l|_{\LL^2}^2 &+ \! k  \sum_{  l=1}^N |\nabla u(t_l) - \nabla \bU^l|_{\LL^2}^2 \!\Big)
 \leq  C \exp\Big( \!\! 
- \gamma \,   \big| \ln \big[ k+{h^2}{k}^{-1} \big] \big|^{\frac{1}{2}}\Big). 
\end{align}

(iii) Suppose that  $G$ satisfies condition {\bf (G2)}.Then, when $k\to 0$ and $h\to 0$ (without any restriction),
 we have for $\gamma < \kappa_0$
 \begin{align}			\label{speed-gen-add_G2}
\EE\Big(\! \max_{ 0\leq l\leq N } |u(t_l) - \bU^l|_{\LL^2}^2 &+ \! k  \sum_{ l=1}^N |\nabla u(t_l) - \nabla \bU^l|_{\LL^2}^2 \!\Big)
 \leq  C \exp\Big( \!
- \gamma \, 
  \big| \ln \big[ k+h^2  \big] \big|^{\frac{1}{2}}\Big). 
\end{align}
\end{theorem} 
\begin{remark}		\label{rk-add-gne}
(i) The speed of convergence is an increasing function of the viscosity $\nu$, and a decreasing function
of the ``strenght" of the noise $K_0\, {\rm Tr }Q$, and of  the length of the time interval.

 (ii) If $\bU^0$ is deterministic such that $|u_0-\bu^0|_{\LL^2} \leq C h^2$ for ``small" $h$, the upper estimates \eqref{speed-gen-add}--\eqref{speed-gen-add_G2} 
 hold for $\gamma < \tilde{\alpha}_0 \frac{4}{\bar{C}^2} \, \sqrt{\frac{\nu^3}{T}}$. 
\end{remark}
\begin{proof}
(i) Let $\eta=k+h^2$, $\varphi(\eta)=k+h^2+ h^2\EE\big( k\sum_{l=1}^N |\nabla \pi^l|_{L^2}^2\big)$, {\color{magenta}  $\psi(\eta)=k\leq \eta$}, 
 and suppose that $\varphi(\eta)\to 0$ as $\eta\to 0$.
Let $X(\eta)$, $Y(\eta)$ and $Z(\eta)$ be defined by \eqref{XYZ}.  Then, as in the proof of Theorem \ref{th-str-gen-mul}, 
  \eqref{eq6.1} and \eqref{eq6.1Bis} are satisfied with  $C_1= \frac{\bar{C}^4\, T}{2^4\nu^3}$ and $a=2$, 
  while \eqref{eq6.2} and \eqref{eq6.2Bis} hold with  $p=2^{q_0-1}$.  Furthermore, Theorem \ref{exp-mom} implies that \eqref{eq6.5} holds for $\tilde{\alpha}_0$
  defined above. 

Therefore, using part (ii) if Theorem \ref{general_strong} and Corollary \ref{general_cor}, we deduce that 
\begin{align*}
\EE\Big(\! \max_{0\leq l\leq N }& |\bu^l - \bU^l|_{\LL^2}^2  + \! k  \sum_{ l=1}^N |\nabla \bu^l - \nabla \bU^l|_{\LL^2}^2 \!\Big) \\
& \leq  C \exp\Big( \!
- \gamma \, 
  \Big| \ln \Big[ k+h^2+h^2\EE\Big( k\sum_{j=1}^N |\nabla \pi^j|_{L^2}^2 \Big) \Big] \Big|^{\frac{1}{2}}\Big)
\end{align*}
holds for $\gamma < \frac{2^{q_0-1}-1}{2^{q_0-1}}\, \tilde{\alpha}_0 C_1^{-\frac{1}{2}}$.
Coupling this upper estimate with \eqref{speed_add_Euler}, we deduce \eqref{speed-gen-add}. Note that unlike the convergence for the time
discretization $\bu^l$ to the true solution $u(t_l)$ described in Theorem \ref{th_Euler_exp}, the final result does not provide a polynomial speed. This
is due to the fact that in Proposition \ref{loc-conv}, the localization involves the fourth power of  $\|\bu\|_V$.

(ii) As in the proof of Theorem  \ref{th-str-gen-mul}, the inequality  \eqref{pressure-W} implies $h^2 \EE\big( \sum_{j=1}^N |\nabla \pi^j|_{L^2}^2 \big) \leq C(T) h^2/k$
Therefore, given $k$ and  $h^2 k^{-1}$ small enough,  
\eqref{speed-gen-add} implies \eqref{speed-gen-add_Bis}

(ii) If $G$ satisfies {\bf (G2)}, the inequality \eqref{pressure-V} implies $k+h^2+ h^2 \EE\big( \sum_{j=1}^N |\nabla \pi^j|_{L^2}^2 \big) \leq C(T)(k+h^2)$. 
Thus,  \eqref{speed-gen-add} implies  \eqref{speed-gen-add_G2}. 
\end{proof} 
\begin{coro}		\label{cor_add_G1}
If $u_0$ is random, independent of $W$ and such that $\EE\big[ \exp\big( \gamma_0 \|u_0\|_V^2\big) \Big]<\infty$, then the statement of Theorem
\ref{speed-gen-add_Bis} is valid with $\tilde{\alpha}_0$ replaced by $\tilde{\beta}_0:= \tilde{\alpha}_0 \frac{\gamma_0}{\gamma_0+\tilde{\alpha}_0}$. 
Indeed, it suffices to use \eqref{exp-moments-u_al}  in the proof of Theorem \ref{speed-gen-add_Bis}. 
\end{coro} 
 \subsection{Strong convergence of Algorithm 2} 
In this subsection, we focus on Algorithm 2,  that is divergence-free finite elements,  and suppose that $\VV_h\subset V$. 
We will state the convergence results for both  multiplicative and additive  stochastic perturbations.

 \subsubsection{Multiplicative noise} 
 
In this case,  the localized convergence result from section \ref{s4bis} only involves the square
 of the $V$ norm of the time discretization. The following result is similar to Theorem \ref{th-str-gen-mul}. 
 \begin{theorem}		\label{th-str-free-mul}
 Suppose that the finite elements are divergence free, so that $\VV_h\subset V$.
 Fix an integer $q_0\geq 3$;  suppose that $u_0\in L^{2^{q_0}}(\Omega;V)$ and $U_0\in L^{2^{q_0}}(\Omega;L^2_{per})$. Let the
coefficient $G$ satisfy the assumptions {\bf (G1)}. Let $\bar{C}$ be the constant defined in \eqref{interpol}. Then if  $k,l$ are ``small"
\begin{align}			\label{speed-free-mul}
\EE\Big( \max_{ 0\leq l\leq N } |u(t_l) - \bU^l|_{\LL^2}^2 + k  \sum_{  l=1}^N |\nabla u(t_l) - \nabla \bU^l|_{\LL^2}^2 \Big)
 &\leq C(\nu,T) 
\big|  \ln \big(k+h^2\big)  \big|^{-(2^{q_0-1}-1)}. 
\end{align}
 \end{theorem}
 \begin{remark}
 Once more, if $u_0$ and $\bU^0$ are deterministic or have moments of all order, for example are Gaussians, the exponent of the logarithm
 is arbitrary large.
 \end{remark}
 \begin{proof}
Let $\eta=k+h^2$,  $\varphi(\eta)=\eta$  and $\psi(\eta)=k\leq \eta$.  Let $X(\eta)$, $Y(\eta)$ and $Z(\eta)$ be defined by \eqref{XYZ}; then
Proposition \ref{loc-cv-divfree} implies that the upper estimate \eqref{eq6.1} is satisfied with $a=1$ and $C_1=\big[ \frac{\bar{C}^2}{4\nu} +1\big] T$.  
Using the upper estimates  \eqref{C(T,q)},   \eqref{U2} and \eqref{sum_grad}, we deduce that \eqref{eq6.2} and \eqref{eq6.2Bis} are true for 
$p=2^{q_0-1}$, while \eqref{C(T,q)}  implies that \eqref{eq6.3} holds for
$q=2^{q_0-1}$.  Therefore, parts (i) of Theorem \ref{general_strong} and Corollary \ref{general_cor} imply
\[ \EE\Big( \max_{ 0\leq l\leq N } |u(t_l) - \bU^l|_{\LL^2}^2 + k  \sum_{  l=1}^N |\nabla u(t_l) - \nabla \bU^l|_{\LL^2}^2 \Big)
 \leq C(\nu,T) 
\big|  \ln \big(k+h^2\big)  \big|^{-(2^{q_0-1}-1)}. 
\]

  Since $k \leq  k+h^2 << e^{-1}$, we have $\big\{ \ln N\big\}^{-(2^{q-1}-1)} \leq  C 
  \big|  \ln \big(k+h^2\big)  \big|^{-(2^{q-1}-1)}$; the upper estimate \eqref{speed_lin_growth_Euler} 
 completes the proof of \eqref{speed-free-mul}. 
 \end{proof} 

\subsubsection{Additive noise }

Assume condition {\bf (G1)}
 holds with $K_1=0$; then the strong speed of convergence is polynomial. 
 \begin{theorem}		\label{th-str-free-add}
 Let $G$ satisfy condition {\bf (G1)} with $K_1=0$, i.e.,  $\| G(u)\|_{{\mathcal L}(K,\WW^{1,2})}^2 \leq K_0$.
Let  
 $u_0\in V$ be deterministic,  
and  $U_0\in L^{2^{q_0}}(\Omega;L^2_{per})$ for some $q_0\in [3,\infty)$.  Set 
\begin{equation}		\label{alpha0-C1}
  \tilde{\alpha}_0:= \frac{\nu}{ 2\, \tilde{C}  K_0 \mbox{\rm Tr}\, Q},  \quad C_0:= \frac{\bar{C}^2 T}{2\nu}\quad
t \mbox{\rm and } \; C_1=: \Big[\frac{ \bar{C}^2}{4\nu}+1\Big] T,
 \end{equation}
  where  $\bar{C}$ is defined by \eqref{interpol} and
$\tilde{C}$ is defined by $|\nabla u|_{\LL^2}^2 \leq \tilde{C} |Au|_{\LL^2}^2$. 

Define the critical exponents as follows: 
\begin{equation}		\label{beta-kappa}
 \beta_0:= \frac{1}{2} \Big( \frac{ \tilde{\alpha_0}}{\tilde{\alpha}_0 + C_0}\Big)  \quad  \mbox{\rm and}\;\,  \kappa_0:=  
\frac{\tilde{\alpha}_0 \big( 2^{q_0-1}-1)}{\tilde{\alpha}_0 \big( 2^{q_0-1}-1) + C_1 2^{q_0-1}}, 
\end{equation} 
 Then for $\gamma_1 < \kappa_0$ and $\gamma < \beta_0$, 
setting $k:=\frac{T}{N}$,
we have for $N$ large enough and small $h$
\begin{equation}			\label{speed-free-add}
\EE\Big(\! \max_{ 0\leq l\leq N} |u(t_l) - \bU^l|_{\LL^2}^2 + k\sum_{l=1}^N |\nabla u(t_l) - \nabla \bU^l|_{\LL^2}^2\ \Big) \leq C  
\big[ \big( k+h^2\big)^{\gamma_1} +  k^\gamma\big] .
\end{equation}
 \end{theorem} 
 \begin{remark}
 The exponent $\beta_0$ (resp. $\kappa_0$),  which reflects the speed for the time (resp. the finite elements) approximation,  is an increasing function
  of the viscosity and decreasing functions of
 $K_0 {\rm Tr }Q\,$. For a given viscosity, as $K_0 {\rm Tr}\, Q \to 0$,  $\kappa_0$ converges to 1. 
 This limit is twice the corresponding one $\frac{1}{2}$
 of $\beta_0$; this is consistent with the scaling between the space and time regularity
  for the heat kernel, which is behind this model. Note that these upper bounds are also approached for a given noise if the viscosity is ``large". 
 The exponent $\beta_0$  cannot be better than the time regularity. 
 The maximal exponents cannot be improved.   \end{remark}
 \begin{proof}
 Let $\eta=k+h^2$,  $\varphi(\eta)=\eta$ and $\psi(\eta)=k\leq \eta$.   As in the proof of Theorem \ref{th-str-free-mul}, 
 \eqref{eq6.1} and \eqref{eq6.1Bis} are satisfied 
 for $a=1$ and $C_1$ defined in \eqref{alpha0-C1}. Furthermore, \eqref{eq6.2} and \eqref{eq6.2Bis} hold for $p=2^{q_0-1}$. Since the noise is additive, \eqref{eq6.5}
 is satisfied for the parameter $\tilde{\alpha}_0$ defined in \eqref{alpha0-C1}. Using parts (iii) of Theorem \ref{general_strong}  and Corollary \ref{general_cor},
 we deduce  
 \[ \EE\Big(\! \max_{ 0\leq l\leq N} |\bu^l - \bU^l|_{\LL^2}^2 + k\sum_{l=1}^N |\nabla \bu^l- \nabla \bU^l|_{\LL^2}^2\ \Big) \leq C  
\big( k+h^2\big)^{\gamma_1} ,\]
 for $\gamma_1<\kappa_0$, where $\kappa_0$ is defined in \eqref{beta-kappa}. Coupling this upper estimate with \eqref{speed_add_Euler}, we deduce
 \eqref{speed-free-add}; this completes the proof. 
 \end{proof} 
  \begin{coro}		\label{cor_add_divfree}
If $u_0$ is random, independent of $W$ and such that $\EE\big[ \exp\big( \gamma_0 \|u_0\|_V^2\big) \Big]<\infty$, then the statement of Theorem
\ref{th-str-free-add} is valid with $\tilde{\alpha}_0$ replaced by $\tilde{\beta}_0:= \tilde{\alpha}_0 \frac{\gamma_0}{\gamma_0+\tilde{\alpha}_0}$. 
Indeed, it suffices to use \eqref{exp-moments-u_al}  in the proof of Theorem \ref{th-str-free-add}. 
\end{coro}

\section{Appendix}
In this section, for an additive noise, we prove the existence of exponential moments of the $V$ norm for the time discretization $u_N(t_l)$ uniformly
on the time grid $\{ t_l=\frac{lT}{N}, \, l=0, \cdots, N\}$ and with a bound independent of $N$.
 This is similar to a similar result proved in \cite{HaiMat} for the solution $u$ to
the stochastic 2D Navier-Stokes equation (see also \cite{Be-Mi_time}). 
\subsection{Deterministic initial condition} We first suppose that $u_0\in V$ is deterministic. 
\begin{theorem}		\label{exp-mom}
Let $G$ satisfy condition  {\bf (G1)}  with $K_1=0$, that is $\| G(u)\|_{{\mathcal L}(K,\WW^{1,2})}^2 \leq K_0$ and
set  $\tilde{\alpha}_0:= \frac{\nu}{ 2\, \tilde{C} \, K_0\,  {\rm Tr} Q}$, where the constant $\tilde{C}$ is defined by
$|\nabla u|_{\LL^2}^2 \leq \tilde{C} |Au|_{\LL^2}^2$. 
 
Let $u_0\in V$; 
   then  for $0<\alpha < \tilde{\alpha}_0$, there exists a positive constant $C(\alpha)$ such that for $N$ large enough, 
\begin{equation} 	\label{exp-moments-ul}
\EE\Big[ \exp\Big( \alpha  \max_{0\leq l\leq N} \| u_N(t_l)\|_V^2 
 \Big) \Big] 
= C(\alpha)<\infty. 
\end{equation} 
\end{theorem}
\begin{proof}
As in sections \ref{s4} and \ref{s5}, for fixed $N$ to ease notations let $\bu^l=u_N(t_l)$, $l=0, \cdots, N$. 
Since Lemma \ref{moments_uN} implies that $\EE\big( \frac{T}{N}\sum_{l=1}^N |A\bu^l|_{\LL^2}^2\big) <\infty$, using integration by parts
we may write 
 \eqref{full-imp1} with $\phi = A\bu^l$; this yields a.s. 
\[ (\bu^l-\bu^{l-1}, A\bu^l) + \frac{T}{N}\Big[ \nu (A \bu^l,  A\bu^l) +  \langle B(\bu^l, \ \bu^l),\, A\bu^l\rangle -  (\pi^l, {\rm div}
A\bu^l) \Big]= \big( G(\bu^{l-1}) \Delta_lW, \, A\bu^l\big).\]
Using \eqref{A-B},  ${\rm div}\, \bu^l=0$, and  integration by parts  in the stochastic term, we deduce
\[ \big( \nabla \bu^l - \nabla \bu^{l-1}, \, \nabla \bu^l\big) + \nu \frac{T}{N} |A\bu^l|_{\LL^2}^2
= \big( \nabla[ G(\bu^{l-1}) \Delta_l W]\, , \, \nabla \bu^l\big).\]
The identity $(a-b,a)=\frac{1}{2} \big( |a|^2 - |b|^2 + |a-b|^2\big)$ implies
\begin{align}   |\nabla  & \bu^l|_{\LL^2}^2 - |\nabla \bu^{l-1}|_{\LL^2}^2 
+ |\nabla ( \bu^l-\bu^{l-1})|_{\LL^2}^2   
\nonumber \\
=&\; 2\big( \nabla[G(\bu^{l-1}) \Delta_l W]\, , \, \nabla \bu^{l}\big)  - 2 \,  \nu\, \frac{T}{N}  |A\bu^l]_{\LL^2}^2  \label{Gl}\\
=&\;  2 \big( \nabla[G(\bu^{l-1}) \Delta_l W]\, , \, \nabla \bu^{l-1}\big) -  2 \frac{T}{N}  \nu |A\bu^l]_{\LL^2}^2  +
2 \big( \nabla[G(\bu^{l-1}) \Delta_l W]\, , \, \nabla [ \bu^l - \bu^{l-1}]\big). \label{Gl-1}
\end{align} 
A similar computation, based on  \eqref{full-imp1} with $\phi = \bu^l$ implies a.s. 
\begin{align}  |  \bu^l|_{\LL^2}^2 &- | \bu^{l-1}|_{\LL^2}^2 
+ | ( \bu^l-\bu^{l-1})|_{\LL^2}^2  =  2\big( G(\bu^{l-1}) \Delta_l W\, , \,  \bu^{l}\big)  - 2 \,  \nu\, \frac{T}{N}  |\nabla \bu^l]_{\LL^2}^2  \label{Gl-Bis}\\
& =\;  2 \big( G(\bu^{l-1}) \Delta_l W\, , \, \bu^{l-1}\big) -  2 \frac{T}{N}  \nu |\nabla \bu^l]_{\LL^2}^2  +
2 \big( G(\bu^{l-1}) \Delta_l W\, , \,  \bu^l - \bu^{l-1} \big). \label{Gl-1Bis}
\end{align} 
For $l=1$, condition {\bf (G1)} with $K_1=0$, the Cauchy-Schwarz and Young inequalities   
imply for $\lambda \in (0,1)$
\[ 2 \big| \big( \nabla[G(\bu^0) \Delta_l W]\, , \, \nabla \bu^{1}\big)\big| \leq \frac{1}{\lambda}
 \| G(\bu^{0})\|_{{\mathcal L}(K,\WW^{1,2})}^2 \|\Delta_1 W\|_K^2
+ \lambda  |\nabla \bu^1|_{\LL^2}^2.\]
Hence \eqref{Gl} implies 
\[ 
|\nabla \bu^1|_{\LL^2}^2  \leq \frac{1}{1-\lambda}  |\nabla u_0|_{\LL^2}^2 +  \frac{1}{\lambda (1-\lambda)}
\, K_0 \|\Delta_1 W\|_K^2 -  2\, \frac{T}{N} \nu |A\bu^1|_{\LL^2}^2 .
\] 
For $l\geq 2$, a similar argument  using the Cauchy-Schwarz and Young inequalities implies
\begin{align*}
2 \big( \nabla[G(\bu^{l-1}) \Delta_lW],& \, \nabla[\bu^l-\bu^{l-1}] \big)  \leq \; 2   | \nabla(G(\bu^{l-1}) \Delta_lW ) |_{\LL^2}
|\nabla (\bu^l-\bu^{l-1})|_{\LL^2} \\
\leq & \;  \|G(\bu^{l-1})\|_{{\mathcal L}(K, \WW^{1,2})}^2 \| \Delta_l W\|_K^2 + | \nabla (\bu^l-\bu^{l-1})|_{\LL^2}^2. 
\end{align*}
Therefore, \eqref{Gl-1}  and condition \eqref{growthG_W} in {\bf (G1)}  with $K_1=0$ imply 
\[   |\nabla \bu^l|_{\LL^2}^2 - |\nabla \bu^{l-1}|_{\LL^2}^2   
\leq 2\big( \nabla[ G(\bu^{l-1}) \Delta_l W], \nabla \bu^{l-1}\big)
- 2\, \frac{T}{N} \nu |A\bu^ l|_{\LL^2}^2 + K_0 \|  \Delta_l W \|_K^2.\]
A similar argument, based on \eqref{growthG_L} in Condition {\bf(G1)} with $K_1=0$,  implies
\begin{align*}  
| \bu^1|_{\LL^2}^2  &\leq \frac{1}{1-\lambda}  | u_0|_{\LL^2}^2 +  \frac{1}{\lambda (1-\lambda)}
\, K_0 \|\Delta_1 W\|_K^2 -  2\, \frac{T}{N} \nu |\nabla \bu^1|_{\LL^2}^2 , \\
 | \bu^l|_{\LL^2}^2 - | \bu^{l-1}|_{\LL^2}^2   
& \leq 2\big(  G(\bu^{l-1}) \Delta_l W,  \bu^{l-1}\big)
 - 2\, \frac{T}{N} \nu |\nabla \bu^ l|_{\LL^2}^2 + K_0 \|  \Delta_l W \|_K^2,\quad l=2, ... N.
 \end{align*}  
Adding these inequalities for $l=1, \cdots, n$, we deduce for $n = 2, \cdots, N$   
\begin{align}		\label{max-nabla-ul}
  \| \bu^n\|_V^2  & 
 \leq \frac{1}{1-\lambda} \|u_0\|_V^2  
+ \frac{2}{\lambda (1-\lambda)} K_0 \sum_{l=1}^n \|\Delta_l W\|_K^2 
 -  2 \nu   \frac{T}{N} 
\sum_{l=1}^n \big( |\nabla \bu^l|_{\LL^2}^2 + |A\bu^l|_{\LL^2}^2 \big) \nonumber\\
& + \sum_{l=2}^n 2 \Big[ \big(  G(\bu^{l-1}) \Delta_l W \, , \,  \bu^{l-1}\big) +
\big( \nabla[ G(\bu^{l-1}) \Delta_l W] \, , \, \nabla \bu^{l-1}\big) \Big]  .
\end{align} 
Let $Y$ be a $K$-valued centered Gaussian random variable with covariance operator $ Q$. 
Using the independence of the time increments $\Delta_l W$
we deduce that for any  $\beta>0$, 
\[ \EE\Big[ \exp\Big(  \beta \sum_{l=1}^N \| \Delta_lW\|_K^2 \Big) \Big] = 
 \Big\{ \EE\Big[ e^{\beta \frac{T}{N} \|Y\|_K^2} \Big]\Big\}^{N} .\]
Proposition 2.16 in \cite{daPZab} implies that if $\tilde{\gamma} \in \big( 0, \frac{1}{2{\rm Tr}\, Q}\big)$ and $\gamma\in (0, \tilde{\gamma})$,
  we have
\[ \EE\big( e^{\gamma \|Y\|_K^2}\big) \leq \exp\Big( \frac{1}{2} \sum_{i=1}^\infty \frac{( 2\gamma)^i}{i} {\rm Tr }\, (Q^i)\Big)  \leq 
\exp\Big( \frac{1}{2} \ln\big(1+ 2\tilde{\gamma} \mbox{\rm Tr}\, Q\big) \Big) <\infty.\]
Hence, if    $\frac{ 2\beta  K_0 T}{\lambda (1-\lambda) \, N }  < \frac{1}{2{\rm Tr}\, Q}$  (which is satisfied for any  $\beta >0$ 
provided that $N$ is large enough), we obtain
 \begin{equation} 		\label{upper_discret}
\Big\{ \EE \Big( \exp\Big( \frac{2 \beta  K_0 T }{\lambda (1-\lambda)\, N }   
\|Y\|_K^2\Big) \Big\}^N \leq  \sqrt{2} .
\end{equation}  
Given $\alpha >0$ and $n=2,\cdots, N$, set
\[ M_n=2\alpha \sum_{l=2}^n \big[ \big( G(\bu^{l-1}) \Delta_lW, \,  \bu^{l-1}\big) + \big( \nabla[G(\bu^{l-1}) \Delta_lW], \, \nabla \bu^{l-1}\big) \big].
\]
Then $(M_n, {\mathcal F}_{t_n}, n=1, ..., N)$ is a discrete martingale. For $s\in [t_l, t_{l+1})$, $l=1, \cdots, N-1$, 
 set $\underline{s}=t_l$ and  $\bu^{\underline{s}}=
\bu^l$.  With these notations, $M_n=\tilde{M}_{t_n}$, where 
 \[
 \tilde{M}_t=2\alpha \int_{t_1}^t \big[  \big(  G(\bu^{\underline{s}}) \, dW(s), \,  \bu^{\underline{s}}\big) + \big( \nabla 
G(\bu^{\underline{s}}) \, dW(s), \, \nabla \bu^{\underline{s}}\big)\big] , \quad t\in [t_1,T]. 
\] 
The processes $G(\bu^{\underline{s}})$, $\nabla G(\bu^{\underline{s}})$,  $\bu^{\underline{s}}$ and  $\nabla \bu^{\underline{s}}$
are ${\mathcal F}_s$-adapted,  and $(\tilde{M}_t, {\mathcal F}_t, t\in [t_1,T])$ is a square integrable martingale, such that
\begin{align}  \label{upper_qua-var} 
\langle \tilde{M} \rangle_{t_n} & \leq \;  4\alpha^2 \int_{t_1}^{t_n}  2\big[   \| G(\bu^{\underline{s}}) \|_{{\mathcal L}(K, \LL^2}^2 \; {\rm Tr}\, Q \, 
| \bu^{\underline{s}}|_{\LL^2}^2 + \| G(\bu^{\underline{s}}) \|_{{\mathcal L}(K, \WW^{1,2})}^2 \; {\rm Tr}\, Q \, 
|\nabla \bu^{\underline{s}}|_{\LL^2}^2 \big]  ds \nonumber  \\
& \leq  \; 8\alpha^2 K_0  {\rm Tr}\, Q \, \tilde{C} \frac{T}{N}  \sum_{l=1}^{n-1} \big[ |\nabla \bu^l|_{\LL^2}^2 +  |A\bu^{l}|_{\LL^2}^2 \big]. 
\end{align}
Using \eqref{max-nabla-ul} we deduce that for $\lambda\in (0,1)$, $\alpha >0$ and $\mu>1$,
\begin{align}		\label{upper_expalpha}
 &\exp\big( \alpha \max_{1\leq n\leq N} \|\bu^l\|_V^2\big) \leq \exp\Big( \frac{\alpha}{1-\lambda} \|u_0\|_V^2\Big) 
\exp\Big( \frac{2\, K_0 \alpha}{\lambda(1-\lambda)}\sum_{l=1}^N \| \Delta_l W\|_K^2 \Big) \\
&\quad \times 
\exp\Big( \max_{2\leq n\leq N}  \big[ M_n - \frac{\mu}{2} \langle \tilde{M} \rangle_{t_n}\big] \Big) 
\exp\Big( \max_{2\leq n\leq N} \Big[  \frac{\mu}{2}  \langle \tilde{M} \rangle_{t_n}
- 2 k \nu \sum_{l=1}^n \big( |\nabla \bu^l|_{\LL^2}^2 +  |A\bu^l|_{\LL^2}^2\big) \Big] \Big) . \nonumber
\end{align} 
Let $\tilde{\alpha}_0=\frac{\nu}{2\,  \tilde{C}\, K_0 {\rm Tr}\, Q}$;    
for $\alpha \in (0, \tilde{\alpha}_0)$, 
we may choose $\mu>1$ such that $\mu \alpha < \tilde{\alpha}_0$; using \eqref{upper_qua-var}
we deduce that for such a choice of $\alpha$ and $\mu$ we have a.s. 
\begin{align*}
\max_{2\leq n\leq N} \Big[  \mu \langle \tilde{M}\rangle_{t_n}  & -  2 k \nu \sum_{l=1}^n \big[ |\nabla \bu^l |_{\LL^2}^2 + |A\bu^l|_{\LL^2}^2 \Big] \\
& \leq 
\Big( \frac{\mu}{2} \, 8 \alpha K_0 \mbox{\rm Tr}\, Q \; \tilde{C} -2 \nu \Big)\; \alpha  \frac{T}{N} \sum_{l=2}^{N-1} \big[ |\nabla \bu^l|_{\LL^2}^2 +
|A\bu^{l}|_{\LL^2}^2 \big] \leq 0.
\end{align*}
Thus, H\"older's inequality with conjugate exponents $\mu$ and $\frac{\mu}{\mu-1}$ implies  for $\lambda = \frac{1}{2}$
\begin{align}		\label{upper_u0det}
 \EE\Big[ \exp\Big( \alpha \max_{0\leq l\leq N} \| \bu^l\|_V^2 \Big) \Big] \leq   &\; \exp\Big( 2\alpha\|u_0\|_V^2 \Big) 
\Big\{ \EE\Big[ \exp\Big(  \frac{ 8 \mu  K_0 \alpha}{\mu-1 }\sum_{l=1}^N \| \Delta_l W\|_K^2 \Big)  \Big] \Big\}^{\frac{\mu-1}{\mu}} \nonumber  \\
& \times \Big\{ \EE \Big[ \max_{2\leq n\leq N} \exp\Big(    \mu \tilde{M}_{t_n} - \frac{\mu^2}{2} \langle \tilde{M} \rangle_{t_n} \Big) \Big]\Big\}^{\frac{1}{\mu}}.
\end{align} 
Since $\{ \exp\big(    \mu \tilde{M}_t - \frac{\mu^2}{2} \langle \tilde{M} \rangle_{t} \big)\}_{t\in [t_1,T]}$ is an exponential martingale, choosing
$N$ large enough to ensure  $\frac{8 \mu \tilde{\alpha}_0 K_0 T}{ (\mu-1)  \, N } < \frac{1}{2 {\rm Tr}\, Q} $, 
\eqref{upper_discret} implies \eqref{exp-moments-ul}. This completes the proof.  
\end{proof} 

The following theorem proves a similar result about the existence of exponential moments for the solution $u$ to the stochastic Navier-Stokes equations.
Its proof, which is similar to the above one, with a slight modification of that of \cite[Lemma3.8]{Be-Mi_time}, is omitted. 
\begin{theorem} 		\label{th_mom_exp_u} 
Let $G$ satisfy condition  {\bf (G1)}  with $K_1=0$, that is $\| G(u)\|_{{\mathcal L}(K,\WW^{1,2})}^2 \leq K_0$ and
set  $\tilde{\alpha}_0:= \frac{\nu}{ 2\, \tilde{C} \, K_0\,  {\rm Tr} Q}$, where the constant $\tilde{C}$ is defined by
$|\nabla u|_{\LL^2}^2 \leq \tilde{C} |Au|_{\LL^2}^2$. 

Let $u_0\in V$; 
   then  for $0<\alpha < \tilde{\alpha}_0$, there exists a positive constant $C(\alpha)$ such that 
\begin{equation} 	\label{exp-moments-udet}
\EE\Big[ \exp\Big( \alpha  \sup_{t\in [0,T]} \| u(t)\|_V^2 
 \Big) \Big] 
= C(\alpha)<\infty. 
\end{equation} 
\end{theorem}

\subsection{Random initial condition}
In this section, we extend Theorems \ref{exp_mom} and \ref{th_mom_exp_u} to a random initial condition $u_0$ such that its $V$ norm has exponential moments.

 \begin{theorem}		\label{th_exp_u}
 Let $G$ satisfy condition  {\bf (G1)}  with $K_1=0$, that is $\| G(u)\|_{{\mathcal L}(K,\WW^{1,2})}^2 \leq K_0$ and suppose that $u_0$ is a $v$-valued
 random variable independent of $W$ and such that $\EE\big[ \exp(\gamma_0 \|u_0\|_V^2) \big]<\infty$ for some $\gamma_0>0$. 
Set   $\tilde{\alpha}_0:= \frac{\nu}{ 2\, \tilde{C} \, K_0\,  {\rm Tr} Q}$, where the constant $\tilde{C}$ is defined by
$|\nabla u|_{\LL^2}^2 \leq \tilde{C} |Au|_{\LL^2}^2$ and $\tilde{\beta}_0:= \tilde{\alpha}_0 \frac{\gamma_0 }{\gamma_0+ \tilde{\alpha_0}}$. 
Then  for $0<\alpha < \tilde{\beta}_0$, there exists a positive constant $C(\beta)$ such that for $N$ large enough, 
\begin{equation} 	\label{exp-moments-u_al}
\EE\Big[ \exp\Big( \alpha  \max_{0\leq l\leq N}   \| u_N(t_l)\|_V^2 
 \Big) \Big]  +\EE\Big[ \exp\Big( \alpha  \sup_{t\in. [0,T]}  \| u(t)\|_V^2\Big)  \Big] 
= C(\beta) < \infty. 
\end{equation} 
\end{theorem}
\begin{remark}  Note that when $u_0$ is deterministic, the exponential moment estimate of $\|u_0\|_V^2$ holds for every $\gamma_0>0$. As $\gamma_0 \to
\infty$, we have $\tilde{\beta}_0 \to \tilde{\alpha}_0$. 
\end{remark}
\begin{proof}
We only prove the exponential moment estimate for $\max_{0\leq l\leq N} \|u_N(t_l)\|_V^2$; the corresponding  one for $\sup_{0\leq t\leq T} \|u(t\|_V^2$ is proven in a
similar (simpler) way.

We want to use H\"older's inequality in \eqref{upper_expalpha} with exponents $p_1\in (1,\infty)$, $p_2=\mu \in (1,\infty)$ and $p_3\in (1,\infty)$ 
such that $\frac{1}{p_1} + \frac{1}{\mu} + \frac{1}{p_3}=1$;  this requires $\alpha \mu < \tilde{\alpha}_0$. 
Suppose that for some $\lambda \in (0,1)$, $\frac{p_1 \alpha}{1-\lambda}\leq \gamma_0$ and $\alpha \mu <\tilde{\alpha}_0$, that is 
$\frac{1}{p_1}>\frac{\alpha}{\gamma_0}$ and $\frac{1}{\mu} > \frac{\alpha}{\tilde{\alpha}_0}$.
 Since $\frac{1}{p_1}+\frac{1}{\mu} <1$,
we deduce that $\alpha < \tilde{\alpha}_0 \frac{\gamma_0}{\gamma_0+\tilde{\alpha}_0}:= \tilde{\beta}_0$. 
For $\alpha < \tilde{\beta}_0$, set $\bar{p}_1=\frac{\gamma_0}{\alpha} $ and $\mu=\frac{\tilde{\alpha}_0}{\alpha}$. 
 Then choose $\lambda \in (0,1)$ small enough 
to have $p_1:= (1-\lambda) \bar{p}_1=\frac{\gamma_0 (1-\lambda)}{\alpha}$ such that $\frac{1}{p_1}+\frac{1}{\mu}<1$, and define $p_3\in (1,\infty)$ by  
$\frac{1}{p_3}=1-\frac{1}{p_1} - \frac{1}{\mu}$. This yields
\begin{align}		\label{upper_u0alea}
 \EE\Big[ \exp\Big( \alpha \max_{0\leq l\leq N} \| \bu^l|_V^2 \Big) \Big] \leq   &\; 
 \Big\{ \EE\Big(  \exp\Big( \frac{p_1 \alpha\|u_0\|_V^2}{1-\lambda} \Big) \Big\}^{\frac{1}{p_1}}
 \Big\{ \EE \Big[ \max_{2\leq n\leq N} \exp\Big(    \mu \tilde{M}_{t_n} - \frac{\mu^2}{2} \langle \tilde{M} \rangle_{t_n} \Big) \Big]\Big\}^{\frac{1}{\mu}}\nonumber \\
& \times \Big\{ \EE\Big[ \exp\Big(  \frac{ p_3 2 K_0 \alpha}{ \lambda(1-\lambda)}\sum_{l=1}^N \| \Delta_l W\|_K^2 \Big)  \Big] \Big\}^{\frac{1}{p_3}} .
\end{align}
Let $N$ be large enough to ensure $\frac{p_3 2 \tilde{\beta}_0 K_0 T}{N \lambda (1-\lambda)} < \frac{1}{2 {\rm Tr}\, Q}$. Then  \eqref{upper_discret}
implies \eqref{exp-moments-u_al}; this concludes the proof. 
\end{proof}

{\bf Acknowledgements} This work started while Hakima Bessaih and Annie Millet stayed at the Mathematisches Forschung
Institute Oberwolfach, during a Research In Pairs program in March 2019.
 They gratefully thank the MFO for the financial support and excellent working conditions. It was completed while A. Millet visited 
 the University of Wyoming in October 2019; she thanks this University for the hospitality an a very friendly atmosphere. 
 A. Millet's research has been conducted within the FP2M federation (CNRS FR 2036).  Hakima Bessaih is partially supported by Simons Foundation grant 582264.

\end{document}